\documentclass[12pt,reqno,twoside]{amsart}
\usepackage{amsmath}
\usepackage{amssymb}
\usepackage[all]{xy}
\usepackage{epsfig}
\usepackage{color}
\usepackage{mathabx}
\usepackage{tikz}
\usetikzlibrary{arrows,calc,matrix}
\usepackage[english]{babel}
\usepackage{mathtools}

\title[Translation of paper of Benoist-Quint]{Translation of the paper 
    `Mesures stationnaires et ferm\'es invariants des espaces
    homog\'enes I', by Yves Benoist and Jean-Fran\c cois Quint,
  Ann. Math. 174 (2011), translated by Barak Weiss}  

\font\sn = cmssi8 scaled \magstep0

\long\def\combarak#1{\ifdraft{\color{red}\sn #1 }\else\ignorespaces\fi}

\numberwithin{equation}{section}

\newif\ifdraft\drafttrue
\draftfalse
\newcommand\name[1]{\label{#1}{\ifdraft{\sn [#1]}\else\ignorespaces\fi}}

\newcommand\eq[2]{{\ifdraft{\ \tt [#1]}\else\ignorespaces\fi}\begin{equation}\label{#1}{#2}\end{equation}}
\newcommand {\equ}[1]{\eqref{#1}}

\newcommand{\MM}{{\mathcal{M}}}

\newcommand{\Q}{{\mathbb {Q}}}

\newcommand{\R}{{\mathbb{R}}}
\newcommand{\TT}{{\mathbb{T}}}

\newcommand{\Z}{{\mathbb{Z}}}

\newcommand{\C}{{\mathbb{C}}}
\newcommand{\BB}{{\mathcal{B}}}
\newcommand{\A}{{\mathcal{A}}}
\newcommand{\CC}{{\mathcal{C}}}
\newcommand{\Gr}{{\mathrm{Gr}}}
\newcommand{\GG}{{\mathcal{G}}}

\newcommand{\N}{{\mathbb{N}}}

\newcommand{\Ad}{{\operatorname{Ad}}}

\newcommand{\GL}{\operatorname{GL}}

\newcommand{\SL}{\operatorname{SL}}

\newcommand{\Lie}{\operatorname{Lie}}

\newcommand{\End}{{\rm End}}

\newcommand {\ignore}[1]  {}

\newcommand{\rank}{{\mathrm{rank}}}

\newcommand{\EE}{{\mathbb{E}}}

\newcommand{\QQ}{{\mathcal Q}}

\newcommand{\PP}{{\mathcal P}}
\newcommand{\XX}{{\mathcal X}}
\newcommand{\YY}{{\mathcal Y}}
\newcommand{\ZZ}{{\mathcal Z}}

\newcommand{\FF}{{\mathcal{F}}}

\newcommand{\til}{\widetilde}

\newcommand{\supp}{{\rm supp}}

\newcommand{\sm}{\smallsetminus}

\newcommand{\vre}{\varepsilon}

\newcommand{\PPP}{\mathbb{P}}

\newcommand{\Hom}{\operatorname{Hom}}

\newtheorem{thm}{Theorem}[section]

\newtheorem{lem}[thm]{Lemma}

\newtheorem{prop}[thm]{Proposition}

\newtheorem{cor}[thm]{Corollary}

\newtheorem{remark}[thm]{Remark}

\newtheorem{dfn}[thm]{Definition}

\begin{document}
\date{\today}
\maketitle


\begin{abstract}
A translation of the famous paper of Benoist and Quint. 

\medskip

The original
abstract in English: 

{\bf Stationary measures and closed invariant subsets of homogeneous
  spaces.} Let $G$
be a real simple Lie group, $\Lambda$ be  a  lattice  of
$G$
and $\Gamma$ be  a  Zariski  dense  subsemigroup  of
$G$. We prove that every infinite $\Gamma$-invariant subset in the quotient
$X =
G / \Lambda$ is dense.
Let $\mu$ 
be a probability measure on $G$
whose support is compact and  spans  a  Zariski  dense  subgroup  of
$G$. We prove that every 
atom free $\mu$-stationary probability measure on $X$ is $G$-invariant.
We also prove similar results for the torus $X=\TT^d$. 
\end{abstract}

\section{Introduction}
The goal of this text is to introduce a new technique in the study of
stationary measures on homogeneous spaces, which we call the
`exponential drift.' 

\subsection{Motivation and principal results}  
We will use it to prove:

\begin{thm}\name{thm: 1.1}
Let $G$ be a connected  almost simple real Lie group, $\Lambda$ a
lattice in $G$, $X= G/\Lambda$ and $\mu$ a probability measure on $G$
with compact support, such that $\supp \, \mu$  generates a
Zariski-dense subgroup of $G$. Then any non-atomic $\mu$-stationary Borel
probability measure on $X$ is the Haar measure on $X$. 
\end{thm}
We now explain some of the (well-known) terminology used in the
statement above. A real Lie group is {\em almost simple} if its Lie
algebra is simple. A probability measure $\nu$ on $X$ is called {\em
  $\mu$-stationary} if $\nu = \mu * \nu.$ It is called {\em non-atomic}
if $\nu(\{x\})=0$ for any $x \in X$. In case $G$ is not a linear
group, when we say that $\Gamma$ is {\em Zariski dense} we mean that
$\Ad(\Gamma)$ is Zariski dense in the linear group $\Ad(G)$ (where
$\Ad: G \to \GL(\mathfrak{g})$ is the adjoint representation). By {\em
  Haar measure} on $X$ we mean the unique $G$-invariant probability
measure on $X$ induced by the Haar measure of $G$. 

This theorem verifies a condition of {\em stiffness} of group actions
introduced by Furstenberg \cite{ref 14} . 

Ratner's theorems describe the measures on homogeneous spaces
invariant and ergodic under a connected group generated by unipotents,
as well as the orbit-closures. Shah and Margulis raised the
question of extending these results to disconnected groups. We deduce
an extension of Ratner's results for Zariski dense subgroups $\Gamma$,
namely: 

\begin{cor}\name{cor: 1.2}
Let $G, \Lambda, X, \Gamma$ be as in Theorem \ref{thm: 1.1}. Then:
\begin{itemize}
\item[a)] Any $\Gamma$-invariant non-atomic measure $\nu$ is the
 Haar measure on $X$. 
\item[b)] Any closed $\Gamma$-invariant infinite set $F \subset X$ is
  equal to $X$. 
\item[c)] Any sequence of distinct finite $\Gamma$-orbits $X_n \subset
  X$ is equidistributed with respect to the Haar measure on $X$. 
\end{itemize}
\end{cor} 

A closed subset $F \subset X$ is {\em $\Gamma$-invariant} if for any
  $\gamma \in \Gamma$, $\gamma F \subset F.$ Point c) means that the
  sequence of measures $\nu_n := \frac{1}{\# X_n} \sum_{x \in X_n}
  \delta_x$ converges to the Haar measure on $X$ with respect to the
  weak-* topology. The simplest example in which one can apply the
  above results is for $G = \SL_d(\R), \, \Lambda =\SL_d(\Z), \, d
  \geq 2,$ with $\mu = \frac{1}{2}( \delta_{g_1} + \delta_{g_2})$
  where the semigroup $\Gamma$ generated by $g_1, g_2$ is
  Zariski-dense. The space $X$ is then the space of unimodular
  lattices in $\R^d$. Part c) generalizes previous results on
  equidistribution of Hecke orbits, obtained by Clozel-Oh-Ullmo. 

Our method can be adapted to handle a larger class of homogeneous
spaces. For instance, it makes it possible generalize a result of
Bourgain, Furman, Lindenstrauss on Mozes as follows (in \cite{ref 3} the
existence of proximal elements was assumed): 

\begin{thm} \name{thm: 1.3}
Let $\Gamma$ be a sub-semigroup of $\SL_d(\Z)$ acting on $\R^d$
strongly irreducibly. Let $\mu$ be a measure on $\SL_d(\Z)$ whose
finite support generates $\Gamma$. Then any non-atomic $\mu$-stationary
probability measure on $X=\TT^d$ is the Haar measure of $X$. 
\end{thm}

Recall that the action of $\Gamma$ on $\R^d$ is called {\em strongly
  irreducible} if any finite index subgroup of the group
generated by $\Gamma$, acts irreducibly on $\R^d$. Note that in case a 
$\mu$-stationary measure $\nu$ is atomic, it can be separated into a
non-atomic and purely atomic part, and both measures in this
decomposition are also $\mu$-stationary. Thus applying Theorem
\ref{thm: 1.1} or Theorem \ref{thm: 1.3} we see that the non-atomic
part is Haar. Regarding the purely atomic part of $\nu$, we will see
(see Lemma \ref{lem: 8.3}) that it is a sum of a family of finitely
supported $\mu$-stationary measures. 

\begin{cor}\name{cor: 1.4}
Let $\Gamma$ be a subsemigroup of $\SL_d(\Z)$ acting strongly
irreducibly on $\R^d$. Then:
\begin{itemize}
\item[a)] The only non-atomic $\Gamma$-invariant probability measure on
  $X$ is the Haar measure. 
\item[b)] The only closed $\Gamma$-invariant infinite subset $F
  \subset X$ is equal to $X$. 
\item[c)] Any sequence of distinct finite $\Gamma$-invariant sets
  $X_n$ becomes equidistributed in $X$ with respect to Haar measure. 

\end{itemize}
\end{cor}

Assertion b) in Corollary \ref{cor: 1.4} is due to Muchnik and to
Guivarc'h-Starkov. 

The approach of \cite{ref 3} is based on a delicate study of the
Fourier coefficients of $\nu$. Our approach is purely
ergodic-theoretic. For that reason it can be readily generalized to
the case of homogeneous spaces. For example, Theorem \ref{thm: 1.1}
and Corollary \ref{cor: 1.2} can be extended, with no significant
change to the proof, to $p$-adic Lie groups $G$. 

\subsection{Strategy}
Our approach is based on the study of the random walk on $X=G/\Lambda$
(resp. $X = \TT^d$) induced by the random walk with law $\mu$ on the
group $G$ (resp., on $\SL_d(\Z)$). In order to study the random walk
we introduce a non-invertible dynamical system which we denote
$\left (B^{\tau, X}, \BB^{\tau, X}, \beta^{\tau, X}, T_\ell^{\tau, X}
\right)$.  Without entering into too many details, we note that this
dynamical system is fibered \combarak{measurable analogue of a fiber
  bundle?}, with fiber $X$, over a suspension $\left( 
  B^\tau, \BB^\tau, \beta^\tau, T^\tau\right)$ of a Bernoulli shift
associated to $\mu$, and thus the space $B^{\tau, X}$ is the
\combarak{(measurably) } product
$B^\tau \times X$. The idea of using such a suspension was inspired by
a paper of Lalley \cite{ref 22}. 

This dynamical system has two properties. Firstly, very simple
formulas express the conditional expectation $\phi_\ell := \EE(\varphi |
\QQ^{\tau,X}_\ell)$ of a bounded $\BB^{\tau, X}$-measurable function
$\varphi$ on $B^{\tau, X}$ relative to the $\sigma$-algebra
$\QQ^{\tau,X}_\ell = \left( T^{\tau,X}_\ell\right)^{-1}\BB^{\tau, X}$
of events after a time $\ell$. Secondly, one has good control of the
norm of products of elements of $G$ associated with words appearing in
these formulas of conditional expectation. In order to construct this
dynamical system, one uses various classical theorems about random
walks due in large part to Furstenberg: positivity of the first
Lyapunov exponent, proximality of the walk induced on the flag
variety, existence of limit probabilities $\nu_b$ for the
probabilities obtained as the image of the stationary measure $\nu$
under a random word $b$. 

Our main argument, which we call the {\em exponential drift}, is
reminiscent of Ratner's idea which uses the Birkhoff ergodic theorem,
replacing that theorem with Doob's Martingale convergence theorem. Its
use was inspired by a paper of Bufetov \cite{ref 4}. This theorem
allows us to assert that the sequence $\varphi_{c, \ell}$ converges,
for $\beta^{\tau, X}$-a.e. $(c,x)$ in $B^{\tau,X}$, to
$\varphi_\infty(c,x)$ where $\varphi_\infty = \EE \left(\varphi |
\QQ_\infty^{\tau, X} \right)$ is the conditional expectation of
$\varphi$ with respect to the tail $\sigma$-algebra $\QQ_\infty^{\tau,
  X} = \bigcap_{\ell \geq 0} \QQ_{\ell}^{\tau, X}$. The idea is to
compare $\varphi_\ell(c,x)$ and $\varphi_\ell(c,y)$  for two points
$x,y$ which are very close to each other and carefully chosen for the
time $\ell$. 

In order to start the drift argument, it is necessary to show that one
may choose, when $\nu$ is non-atomic, two points $(c,x)$ and $(c,y)$
which are not on the same stable leaf relative to the factor
$B^{\tau, X} \to B^\tau$. This is a crucial point in our argument. It
shows, roughly speaking, that the relative entropy of the fibered
system is nonzero. In order to demonstrate this we exhibit a
recurrence phenomenon for the random walk on $X$, analogous to the
work of Eskin and Margulis \cite{ref 9}, and combine this phenomenon
with the ergodic theorem of Chacon-Ornstein.  

In order to develop our exponential drift argument, it is necessary to
obtain good control of norms of products of random matrices with law
$\mu$, in the vector space $V = \Lie(G)$ (resp. $V= \R^d$). The
existence, due to Furstenberg, of an attracting limit subspace $V_b$
is very useful. 

When applying our drift argument, work remains. Unlike Ratner's
argument, our argument only yields very patchy invariance
properties for the stationary measures. For this reason we introduce
a function which associates to each point $(c,x)$, a conditional
measure $\sigma(c,x)$ of the limit probability $\nu_c$ along the
foliation given by some limit subspace $V_c$. We identify all the
spaces $V_c$ thus constructed with the action of a unique vector space
$V_0$, an action which we call the {\em horocyclic flow} and denote by
$\Phi_v$. This point is important because it makes it possible to
consider, as in \cite{ref 8}, the function $\sigma$ as a map
taking values in a fixed vector space, the space of Radon measures on
$V_0$  up to normalization. It is this map $\sigma$ to which
we apply our drift argument. A crucial point is that the map
$\sigma$ is $\QQ^{\tau, X}_\infty$-measurable. This results in
commutation relations between $\Phi_v$ and $T^{\tau, X}_\ell$,
relations analogous to those existing in the hyperbolic plane between
the geodesic and horocyclic flow. 

The drift argument implies that the connected component $J(c,x)$ of
the stabilizer of $\sigma(c,x)$ in $V_0$ is almost surely
nontrivial. This makes it possible to view the probability $\nu_c$,
and hence $\nu$, as an average of probabilities $\nu_{c,x}$ which are
invariant under a nontrivial subspace $J(c,x)$ of $V_0$. 

In the case of the torus, one then deduces that the probabilities
$\nu_{c,x}$, and hence $\nu$, are averages of probability measures
supported on nontrivial subtori. Since the support of $\mu$ acts
strongly irreducibly on $\R^d$, $\nu$ is necessarily the Haar measure
on $\TT^d$. 

In the case of a homogeneous space, an application of Ratner's
theorems makes it possible to express $\nu_{c,x}$ as an average of
probability measures supported on orbits of nontrivial closed connected
subgroups $H$ of $G$. The 
$G$-invariance of $\nu$ is deduced, thanks to a phenomenon of
non-existence of $\mu$-stationary measures on the homogeneous space
$G/H$ with unimodular non-discrete stabilizer. 

It is remarkable that our drift argument works even without it being
necessary to explicitly describe the tail $\sigma$-algebra
$\QQ_\infty^{\tau,X}$. However, we will describe this tail
$\sigma$-algebra in a forthcoming work and employ to this end the
works of Blanchard, Conze, Guivarc'h, Raugi and Rohlin. 

\subsection{Structure of the paper}
Chapters 2-5 collect the constructions and the properties of the
dynamical systems associated with random walks that we will need. 

Chapters 6-8 are devoted to the study of stationary measures on the
spaces $X = G/\Lambda$ and $X = \TT^d$. These two cases will be
treated simultaneously. We suggest to the reader to focus primarily on
the case that $X$ is the torus $\TT^2$. Almost all of the arguments we
shall develop are indispensable even for this case. 

The goal of chapter 2 is formulas for the conditional expectation of
fibrations and suspensions over non-invertible dynamical systems,
including the remarkable `law of the last jump.' Chapter 3 deals with
some properties of stationary measures on Borel spaces equipped with a
Borel action: existence of limit measures and the very useful
phenomenon of recurrence off the diagonal. In chapter 4 we recall the
construction of conditional measures \combarak{usually called
  leaf-wise measures} along the orbits of a Borel
action with discrete stabilizers. In chapter 5 we study linear
strongly irreducible random walks. We recall the results of
Furstenberg and introduce the dynamical system $\left(B^\tau,
  \BB^\tau, \beta^\tau, T^\tau \right)$ which is a suspension over a
Bernoulli shift. 

In chapter 6 we introduce the fibered dynamical system $\left(B^{\tau,X},
  \BB^{\tau,X}, \beta^{\tau,X}, T^{\tau,X} \right)$ associated to the
random walk on $X=G/\Lambda$ or $X = \TT^d$. We check that this random
walk satisfies not only the properties of recurrence off the diagonal
which we will need in order to initiate the drift, but also the
recurrence outside finite orbits which we will need in order
to obtain topological consequences.  We will also show non-existence
of stationary measures on certain homogeneous spaces of semi-simple
Lie groups, which we will require at the end of our study, for the
space $X= G/\Lambda$. At the end of the chapter we will introduce the
horocyclic flow $\Phi_v$ on $B^{\tau, X}$ and the conditional horocyclic
map $\sigma$, and check that $\sigma$ is $Q_{\infty}^{\tau,
  X}$-measurable. 

In chapter 7 we will present our general drift argument, and apply it to the
map $(c,x) \mapsto \sigma(c,x)$. In section 8 we exploit the
invariance properties of stationary measures, which follow from the
drift argument, enabling us to conclude the proofs of Theorems
\ref{thm: 1.1} and \ref{thm: 1.3}. We then easily deduce Corollaries
\ref{cor: 1.2} and \ref{cor: 1.4}.  

\subsection{Acknowledgements}
We thank Y. Hu, F. Ledrappier, H. Oh, R. Spatzier and J.-P. Thouvenot
for interesting discussions of this subject and the Brown University
Dept. of Mathematics for
its hospitality.

\section{Suspensions and extensions}
{\em The goal of this chapter is to obtain formulas for the
  conditional expectations with respect to the tail $\sigma$-algebras
  in suspensions and fibrations over non-invertible
  dynamical systems (Proposition \ref{prop: 2.3} and Lemma \ref{lem:
    2.5}). }

\medskip

\subsection{Cohomologous functions}
{\em The following lemma makes it possible to restrict our attention
  to suspensions with positive roof functions.}

\begin{lem}\name{lem: 2.1}
Let $\left(B, \BB, \beta \right)$ be a Lebesgue probability space,
equipped with an ergodic measure preserving transformation $T$. Let
$\theta: B \to \R$ be an integrable function (that is $\int_B |\theta|
d \beta < \infty$) with $\int_B \theta
d\beta >0$. Then there is a positive function $\varphi$ which is 
almost surely finite, and a positive integrable function $\tau$, such
that 
$$
\theta - \varphi \circ T + \varphi = \tau.
$$
The function $\tau$ can be chosen to be bounded below by a constant
$\vre_0>0$. The function $\tau$ can be chosen to be bounded if
$\theta$ is bounded. 
\end{lem}

In other words, the function $\theta$ is cohomologous to $\tau$ via
$\varphi.$ 

\begin{proof}
For $p \geq 1$, denote $\theta_p = \theta + \theta \circ T + \cdots + \theta \circ
T^{p-1}$ and let
$$
\psi = \inf_{p \geq 1} \theta_p, \ \ \tau = \max (\psi, 0),  \ \
\varphi = -\min(\psi, 0).
$$ 
By the Birkhoff ergodic theorem, for $\beta$-a.e. $b$ in $B$,
$\theta_p (b) \to_{p \to \infty} \infty$. This implies that for almost
all $b$, the $\inf$
in the definition of $\psi(b)$ is a min and $\varphi(b)$ is
finite. Since $\psi \leq \theta$, we find $\tau \leq \max 
(\theta, 0)$ and hence $\tau$ is integrable. Finally, by definition, 
$$
\tau - \varphi = \psi = \min(\theta, \theta+\psi \circ T) = \theta +
\min(0, \psi \circ T) = \theta - \varphi \circ T.
$$
In order to obtain $\tau$ which is bounded below by $\vre_0$, apply
the previous reasoning to the function $\theta - \vre_0$. This is
possible whenever $\vre_0 < \int_B \theta d\beta$. The function $\tau$
given in the construction is bounded when $\theta$ is. 
\end{proof} 

\subsection{Suspension of a non-invertible system}
\name{subsec: 2.2}
{\em We define in this section the suspension of a dynamical system
  where the roof function has a factor taking values in a compact group.}

Let $\left(B, \BB, \beta \right)$ be a Lebesgue probability space,
equipped with an ergodic measure preserving transformation $T$. Let $M$
be a compact metrizable topological group and 
$$
\tau = (\tau_\R, \tau_M): B \to \R \times M
$$
a measurable map such that $\tau_\R: B \to \R$ is a positive
integrable function. For any $p \geq 0$, and for $\beta$-a.e. $b$ in
$B$, denote
$$
\tau_{\R, p} = \tau_{\R} (T^{p-1}b) + \cdots + \tau_\R (b)
$$
and 
$$
\tau_{M,p}(b) = \tau_M(T^{p-1}b) \cdots \tau_M(b).
$$
Define the suspension $\left(B^\tau, \BB^\tau, \beta^\tau, T^\tau
\right)$ as follows. The space $B^\tau$ is 
$$
B^\tau  = \{c = (b,k,m) \in B \times \R \times M : 0 \leq k < \tau_\R(b)\},
$$
the measure $\beta^\tau$ is obtained by normalizing the restriction to
$B^\tau$ of the product measure of $\beta$ and the Haar measure of $\R
\times M$, the $\sigma$-algebra $\BB^\tau$ is the product
$\sigma$-algebra, and for almost every $\ell \in \R_+$ and $c =
(b,k,m) \in B^\tau$, 
$$
T^\tau_\ell(c) = \left(T^{p_\ell(c)}b, k+\ell - \tau_{\R,
    p_\ell(c)}(b), \tau_{M, p_\ell(c)}m \right)
$$
where
$$
p_\ell(c) = \max\{ p \in \N: k+\ell - \tau_{\R, p}(b) \geq 0\}.
$$
The flow $T^\tau_\ell$ is then defined for all positive
times. \combarak{It is ergodic.}
\begin{lem}\name{lem: 2.2}
The semigroup $(T^\tau_\ell)$ of transformations of $B^\tau$ preserves
the measure $\beta^\tau$. 
\end{lem}

\begin{proof}
The simplest approach is to avoid all calculations and consider $\left(B^\tau,
  \BB^\tau, \beta^\tau, T^\tau \right)$ as a factor of the suspension
$\left(\til B^\tau,
  \til \BB^\tau, \til \beta^\tau, \til T^\tau \right)$ of the natural
extension  $\left(\til B,
  \til \BB, \til \beta, \til T \right)$ of  $\left(B,
  \BB, \beta, T \right)$ and reduce to the case when $T$ is
invertible. 

When $T$ is invertible, one can identify the suspended dynamical
system as the quotient of the product $B \times \R \times M$ by the
transformation $S: (b,k,m) \mapsto (Tb, k - \tau_\R(b),
\tau_M(b)m)$. The flow $T^\tau_\ell $ is induced by the flow $\til
T^\tau_\ell$ preserving the product measure on $B \times \R \times
M$. Therefore $T^\tau_\ell$ preserves $\beta^\tau$. 
\end{proof}

We remark finally, that it follows from the Birkhoff ergodic theorem,
that for $\beta^\tau$-almost every $c \in B^\tau$, 
\eq{eq: 2.1}{
\lim_{p \to \infty} \frac{1}{p} \tau_{\R, p} (b) = \int_B \tau_\R d\beta,
\ \ \lim_{\ell \to \infty} \frac{1}{\ell} p_\ell(c) = \frac{1}{\int_B
  \tau_\R d\beta}.
}

\subsection{The law of the last jump}\name{subsec: 2.3}
{\em We now establish the law of the last jump which plays a crucial
  role in controlling the drift, in \S \ref{subsec: 7.1}. This law is
  an explicit formula for the conditional expectation of an event in
  $B^\tau$ relative to $(T^\tau_{\ell})^{-1}(\BB^\tau)$ when the base
  system is a Bernoulli shift.
}

Let $\left(A, \A, \alpha\right)$ be a Lebesgue probability space and
$\left(B, \BB, \beta, T\right)$ a one-sided Bernoulli shift on the
alphabet $\left(A, \A, \alpha\right)$, that is $B = A^{\N}, \beta =
\alpha^{\otimes \N}, \BB$ is the product $\sigma$-algebra
$\A^{\otimes \N}$ and $T$ is the right shift which sends $b = (b_0,
b_1, \ldots) \in B$ to $Tb=(b_1, b_2, \ldots)$. Let $M$ be a
metrizable compact topological group, let $\tau = (\tau_\R, \tau_M): B
\to \R \times M$ be a measurable map such that $\tau_\R: B \to \R$ is
positive and integrable, and let $\left(B^\tau, \BB^\tau,
  \beta^\tau, T^\tau \right)$ be the suspension defined in \S
\ref{subsec: 2.2}. 

We will require notation to parameterize the branches of the inverses
of $T^\tau_\ell$. For $q \geq 0$ and $a,b \in B$, we denote by $a[q]$
the beginning of the word $a$ written from right to left as
$a[q]=(a_{q-1}, \ldots, a_1, a_0)$ and $a[q]b \in B$ the concatenated
word 
$$a[q]b=(a_{q-1}, \ldots, a_1, a_0, b_0, b_1, \ldots, b_p, \ldots).$$
For $c = (b,k,m) \in B^\tau$ and $\ell$ in $\R_+$, let $q_{\ell,c}:
B\to \N$ and $h_{\ell, c}: B \to B^{\tau}$ the maps given, for $a \in
B$, by 
$$
q_{\ell, c} = \til q_{\ell, c'} \ \ \mathrm{and} \ \ h_{\ell, c} =
\til h_{\ell, c'}, \ \ \mathrm{where} \ \ c' = T^\tau_\ell(c)
$$ 
and 
$$
\til q_{\ell, c}(a) = \min\{q \in \N: k-\ell+\tau_{\R,q}(a[q]b) \geq 0\},
$$
$$
\til h_{\ell, c}(a) = (a[q]b, k-\ell+\tau_{\R, q}(a[q]b),
\tau_{M,q}(a[q]b)^{-1}m) \ \ \mathrm{with} \ \ q = \til q_{\ell, c}(a).
$$
By Birkhoff's theorem applied to the two-sided shift, for $\beta$
a.e. $a \in B$, and $\beta^\tau$ a.e. $c \in B^\tau$, one has the
equality
$$
\lim_{q \to \infty} \frac{1}{q} \tau_{\R, q}(a[q]b) = \int_B \tau_\R d\beta>0.
$$
\combarak{In the paper the rhs was $\tau$ and not $\tau_\R$}
Hence the function $\til q_{\ell, c}$ is almost surely finite and the
image of the map $\til h_{\ell,c}$ is the fiber
$(T^\tau_\ell)^{-1}(c)$. The function $q_{\ell,c}$ is thus also almost
surely finite. In addition, for $\beta$-a.e. $a \in B$, for every $q
\geq 1$, the function $b \mapsto \tau_{\R,q}(a[q]b)$ is
$\beta$-integrable. Therefore by Birkhoff's ergodic theorem, for
$\beta^\tau$-a.e. $c \in B^\tau$, one has
$$
\lim_{p \to \infty} \frac{1}{p} \tau_{\R,q}(a[q]T^pb)=0
$$
\combarak{here is my interpretation, but using the ergodic theorem
  seems a bit like overkill: if
  $\psi(b) = \tau_{\R, q}(a[q]b)$, then 
\[\begin{split}
\lim_{p \to \infty} \frac{1}{p} \psi(T^pb) -0  & = \lim_{p \to \infty}
\frac{1}{p} ( \psi(T^pb) - \psi(b))  \\ 
&  = \lim_{p \to \infty}
\frac{1}{p}(\sum_1^p \psi(T^ib) - \sum_0^{p-1} \psi(T^ib))   \\ & = \int
\psi d\beta - \int \psi d\beta =0. 
\end{split}\]
}
and hence, by \equ{eq: 2.1}, we have 
\eq{eq: 2.2}{
\lim_{\ell \to \infty} q_{\ell, c}(a) = \infty.
}
\combarak{It seems to me that this follows from an analogue of
  \equ{eq: 2.1} for the shift in the backward direction, and has
  nothing to do with the preceding line. 
That is 
\[\begin{split}
q_{\ell,c}(a) & = \min\{q: \tau_{\R, q}(c') \geq \ell - k\} \\
& = \min\{q: \tau_{\R, q}(a_{q-1} \ldots a_0 T^\tau_\ell c) \geq \ell
- k\} \\
& = \min\{q: \sum_{i=0}^{q-1}\tau_{\R, q}(a_{i} a_{i-1}  \ldots a_0b_{q'(\ell)}
b_{q'(\ell+1)} \ldots ) \geq \ell - k\} 
\end{split}\]
where $q'(\ell)$ satisfies $c' T^{\tau}_\ell(c)= (b', k,m), b' =
T^{q'(\ell)}(b)$. Then by \equ{eq: 2.1} we have that $q'(\ell) \to
\infty$ and therefore $q_{\ell, c} \to \infty$. 
} 

Finally, the image of the map $h_{\ell, c} $ is the fiber of
$T^\tau_\ell$ passing through $c$:
$$
\{c'' \in B^\tau : T^\tau_\ell(c'') = T^\tau_\ell(c)\}, 
$$
that is the atom of $c$ in the partition associated with the
$\sigma$-algebra $(T^\tau_\ell)^{-1}(\BB^\tau)$. 
\begin{prop}\name{prop: 2.3}
The conditional expectation with respect to the $\sigma$-algebra
$(T^\tau_\ell)^{-1}(\BB^\tau)$ is given, for any positive measurable
function $\varphi$ on $B^\tau$ and for $\beta^\tau$-a.e. $c = (b,k,m)
\in B^\tau$, by 
$$
\EE\left( \varphi |(T^\tau_\ell)^{-1}(\BB^\tau) 
 \right) (c) = \int_B \varphi(h_{\ell, c}(a)) d\beta(a).
$$
\end{prop}
In other words, if we regard every element of the fiber of
$T^\tau_\ell$ over a point $c' = (b', k', m') = T^\tau_\ell(c)$ in
$B^\tau$, when completing the infinite word $b'$ by the finite word
$a[q]$ written from right to left, the law of the finite word is
obtained by randomly printing the letters $a_i$, independently with
law $\alpha$ in the alphabet $A$, where printing stops  at time
$q_{\ell, c}(a)$. 

In particular, if $\tau$ is bounded and if $\ell \geq \sup \tau_\R$,
the law of the last jump $a_0$ is $\alpha$. More generally, if $\ell
\geq q \sup \tau_\R$ the law of the last $q$ jumps $(a_{q-1}, \ldots,
a_0)$ is $\alpha^{\otimes q}$. 

\combarak{Don't understand the last sentence. Why does this follow?} 

\begin{proof}
To simplify the notation used in the proof, we assume that $M$ is trivial and thus $\tau =
\tau_\R$. The general case of the proof is the same. 

Introduce the function $\varphi_0(c) = \int_B \varphi(\til
h_{\ell,c}(a)) d\beta(a)$. In order to show that the function
$\varphi_0 \circ T^\tau_\ell$ is the sought-after conditional
expectation, it suffices to show that, for any positive
$\BB^\tau$-measurable function $\psi$, we have the equality 
\eq{eq: 2.3}{
\int_{B^\tau} \psi(T^\tau_\ell c) \varphi(c) d\beta^\tau(c) =
\int_{B^\tau} \psi(T^\tau_\ell c) \varphi_0(T^\tau_\ell c)
d\beta^\tau(c). 
}
To this end, we note that the left-hand side $G$ is equal to 
$$
G= \sum_{p=0}^\infty \int_{B^\tau} \mathbf{1}_{\{p_\ell(c) = p \}}
\psi(T^pb, k+\ell - \tau_p(n)) \varphi(b,k) d\beta(b) dk.
$$
Introduce the variable $c' = (b',k) = (T^pb, k+\ell - \tau_p(b)) \in
B^\tau$ and $a \in B$ such that $a[p]=(b_0, \ldots, b_{p-1}).$ One
finds, when writing $B(c',p) = \{a \in B: \til q_{\ell, c'}(a)=p\}$,
that
$$
G = \int_{B^\tau} \psi(b',k') \sum_{p=0}^\infty \int_{B(c',p)}
\varphi(a[p]b)) d\beta(a) d\beta(b') dk',
$$ 
and hence that 
$$
G = \int_{B^\tau} \psi(c') \int_B \varphi \left(\til h_{\ell, c'}(a) \right)
d\beta(a) d\beta^\tau(c') = \int_{B^\tau} \psi(c') \varphi_0(c') d\beta^\tau(c').
$$
Now \equ{eq: 2.3} follows from the fact that $T^\tau_\ell$ preserves
the measure $\beta^\tau$. 
\end{proof}

\subsection{Conditional expectation for the fibered system} {\em We
  conclude this chapter with a general abstract lemma which constructs an
  invariant probability measure for the fibered dynamical system and
  by calculating its conditional expectation.}

Let $(B, \BB)$ be a standard Borel space, i.e. isomorphic to a
separable complete metric space with its Borel $\sigma$-algebra, and
let $\beta$ be a Borel probability measure on $B$ and $T$ an
endomorphism of $B$ preserving $\beta$. Let $(X, \XX)$ be a standard
Borel space, $\pi: B \times X \to B$ the projection onto the first
factor, and $\hat{T}$ a measurable transformation of $B \times X$ such
that $\pi \circ \hat{T} = T \circ \pi$. Below we will write, for $(b,x)
\in B \times X$, 
$$
\hat{T}(b,x)= (Tb, \rho(b)x).
$$
The space $\PP(X)$ of probability measures on $(X, \XX)$ has itself
the natural structure of a Borel space: this is the structure
generated by the maps $\PP(X) \to \R, \ \nu \mapsto \int_X \varphi
d\nu$, where $\varphi: X \to \R$ is a bounded Borel function. If one
realizes $X$ as a compact metric space endowed with its Borel
$\sigma$-algebra, this structure is generated by the maps $\PP(X) \to
\R, \ \nu \mapsto \int_X \varphi d\nu$ where $\varphi: X \to \R$ is a
continuous function. In particular, with respect to this Borel
structure, the space $\PP(X)$ is also a standard Borel space. 

Consider a $\BB$-measurable collection $B \to \PP(X), \ b \mapsto \nu_b$
of probability measures on $X$ such that for $\beta$-a.e. $b \in B$,
we have 
\eq{eq: 2.4}{
\nu_{Tb} = \rho(b)_*\nu_b.
}
We will denote by $\lambda$ the Borel probability measure on $(B
\times X, \BB \otimes \XX)$ defined by setting, for each positive
Borel function $\varphi: B \times X \to \R_+$, 
$$
\lambda(\varphi) = \int_B \int_X \varphi(b,x) d\nu_b(x) d \beta(b).
$$
We will abbreviate this by writing 
\eq{eq: 2.5}{
\lambda = \int_B \delta_b \otimes \nu_b d\beta(b).
}

\begin{lem}\name{lem: 2.4}
\begin{itemize}
\item[a)] The measure $\lambda$ is $\hat{T}$-invariant and satisfies
  $\pi_*\lambda = \beta$. 
\item[b)]
Conversely, if $T$ is invertible, then any $\hat{T}$-invariant
probability measure on $B \times X$ such that $\pi_*\lambda = \beta$
is given by \equ{eq: 2.5} for some measurable family of probabilities $b
\mapsto \nu_b$ satisfying \equ{eq: 2.4}. 
\end{itemize}
\end{lem}
\begin{proof}
a) The $\hat{T}$-invariance of $\lambda$ can be seen by a simple
computation. For a $(\BB \otimes \XX)$-measurable function $\varphi: B
\times X \to \R_+$, one has
\[
\begin{split}
\int_{B \times X} \varphi(\hat{T}(b,x)) \lambda(b,x) & = \int_B \int_X
\varphi(Tb, \rho(b)x) d\nu_b(x) d\beta(b) \\
& \stackrel{\equ{eq: 2.4}}{=} \int_B \int_X 
\varphi(Tb, x) d\nu_{Tb}(x) d\beta(b) \\ 
& \stackrel{T_*\beta=\beta}{=} \int_B \int_X 
\varphi(b, x) d\nu_{b}(x) d\beta(b) \\ 
& = \int_{B \times X} \varphi(b,x) d\lambda(b,x).
\end{split}
\]
In case $\varphi$ does not depend on the variable $x$, since the
measures $\nu_b$ are probabilities, one has 
\[
\begin{split}
 \int_{B \times X} \varphi(b,x) d\lambda(b,x) & = \int_B \int_X
 \varphi(b) d\nu_b(x) d\beta(b) \\
& = \int_B \varphi(b) d\beta(b).
\end{split}
\]
This implies $\pi_* \lambda = \beta$. 

b) The probability measures $\nu_b$ are the conditional probabilities
of $\lambda$ along the fibers of $\pi$. Since $T$ is invertible,
condition \equ{eq: 2.4} follows from the $\hat{T}$-invariance of
$\lambda$ and uniqueness of conditional probabilities. 
\end{proof}

We quickly recall the theorem of Rohlin \cite{ref 32} about
disintegration of measures, which we will use below, and its
relationship with conditional expectations. 

{\em 
Let $\eta$ be a probability measure on a standard Borel space $(Y,
\YY)$. For any $\sigma$-algebra $\YY'  = p^{-1}(\ZZ) \subset \YY$
corresponding to a Borel factor $p: (Y, \YY) \to (Z, \ZZ)$, we denote
by $y \mapsto \eta^{\YY'}_y \in \PP(Y)$ the disintegration of $\eta$
relative to $\YY'$. This is a $\YY'$-measurable map such that, for
$\eta$-a.e. $y \in Y$, $\eta^{\YY'}_y$ is supported on $p^{-1}(p(y))$
and one has 
\eq{eq: 2.6}{\eta = \int_Y \eta^{\YY'}_y d\eta(y).
}
This map $y \mapsto \eta^{\YY'}_y$ is unique up to a set of
$\eta$-measure zero.

In addition, for any $\YY$-measurable positive function $\varphi: Y
\to \R_+$, for $\eta$ a.e. $y \in Y$, one has 
$$
\EE \left(\varphi | \YY' \right) (y)= \int_B \varphi(y') d\eta^{\YY'}_y (y')
$$ 
}

The following lemma asserts that the disintegration of $\lambda$ with
respect to the factor $\hat{T}: B \times X \to B \times X $ can be
easily derived from the distintegration of $\beta$ with respect to the
factor $T: B \to B$. 
\begin{lem} \name{lem: 2.5}
Assume that for $\beta$-a.e. $b \in B$, the map $\rho(b): X \to X$ is
bijective. Then for every $(\BB \otimes \XX)$-measurable and
$\lambda$-integrable function $\varphi: B \times X \to \C$ and for
$\lambda$-a.e. $(b,x) \in B \times X$, we have 
\eq{eq: sought for}{
\EE\left( \varphi | \hat{T}^{-1} (\BB \otimes \XX) \right) (b,x) =
\int_{B \times X } \varphi(b', \rho(b')^{-1}\rho(b)x)
d\beta_b^{T^{-1}\BB}(b'). 
}
\end{lem}
\begin{proof}
As explained above, for $\lambda$-a.e. $(b,x) \in B \times X$, one has
the equality 
$$
\EE \left(\varphi | \hat{T}^{-1} (\BB \otimes \XX) \right) (b,x)=
\int_{B \times X} \varphi(b',x') d\lambda^{\hat{T}^{-1}(\BB \otimes \XX)}_{(b,x)}(b',x').
$$
Thus it remains to identify the measures $\lambda^{\hat{T}^{-1}(\BB
  \otimes \XX)}_{(b,x)}.$ 

We note first that, since $\rho(b)$ is bijective, for
$\lambda$-a.e. $(b,x) \in B \times X$, the projection $\pi$ induces a
bijection 
of the fiber $\hat{T}^{-1}(\hat{T}(b,x))$ with $T^{-1}(Tb)$ where the
inverse is given by $b' \mapsto (b', \rho(b')^{-1}\rho(b)x)$. Denote
by $\mu_{(b,x)}$ the measure on $B \times X$ given by the right hand
side 
in the sought-for equality \equ{eq: sought for}:
$$
\int_{B \times X} \varphi(b',x') d\mu_{(b,x)}(b',x') = \int_B
\varphi(b', \rho(b')^{-1}\rho(b)x) d\beta^{T^{-1}\BB}_b(b').
$$
We want to show that for $\lambda$-a.e. $(b,x) \in B \times X$, we
have 
$$
\lambda^{\hat{T}^{-1}(\BB \otimes \XX)}_{(b,x)} = \mu_{(b,x)}.
$$
To this end, first note that the map $(b,x) \mapsto \mu_{(b,x)}$ is
$\hat{T}^{-1}(\BB \otimes \XX)$-measurable and that the measure
$\mu_{(b,x)}$ is supported on $\hat{T}^{-1}(\hat{T}(b,x))$. Secondly
we we will compute the following integral $I$ for every
$\lambda$-integrable function $\varphi: B 
\times X \to \C$: 
$$
I = \int_{B \times X} \int_{B \times X} \varphi(b',x') d\mu_{(b,x)}
(b',x') d\lambda(b,x).
$$
For $\beta$-a.e. $b$ we apply Fubini's theorem in the space $(B \times
X, \BB \otimes \XX, \beta^{T^{-1}\BB}_b \otimes \nu_b)$, and obtain
$$
I= \int_B \int_B \int_X \varphi(b',
\rho(b')^{-1}\rho(b)x)d\nu_b(x)d\beta^{T^{-1}\BB}_b(b') d\beta(b). 
$$
Using \equ{eq: 2.4} one finds
\[
\begin{split}
I & =  \int_B \int_B \int_X \varphi(b', \rho(b')^{-1}x) d\nu_{Tb}(x)
d\beta^{T^{-1}\BB}_b(b') d\beta(b) \\
& = \int_B \int_B \int_X \varphi(b', \rho(b')^{-1}x) d\nu_{Tb'}(x)
d\beta^{T^{-1}\BB}_b(b') d\beta(b). 
\end{split}
\]
Finally, applying once more \equ{eq: 2.4} and \equ{eq: 2.6}, one
obtains 
\[
\begin{split}
I & =  \int_B \int_B \int_X \varphi(b',x) d\nu_{b'}(x)
d\beta^{T^{-1}\BB}_b(b') d\beta(b) \\
& = \int_B \int_X \varphi(b, x) d\nu_{b}(x) d\beta(b) = \int_{B \times
  X} \varphi(b,x) d\lambda(b,x). 
\end{split}
\]
By uniqueness of the disintegration, we have the equality
$\lambda^{\hat{T}^{-1}(\BB \otimes \XX)}_{(b,x)} = \mu_{(b,x)}$, for
$\lambda$-a.e. $(b,x) \in B \times X$. 
\end{proof} 

\section{Random walks on $G$-spaces}
{\em In this chapter we collect some fundamental properties of
  stationary measures which are valid in a very general context.}

\subsection{Stationary measures and Furstenberg measure}\name{subsec: 3.1}
{\em To each stationary probability measure $\nu$ we associate a
  probabilistic dynamical system $\left( B^X, \BB^X, \beta^X, T^X \right)$.}

Let $G$ be a metrizable locally compact group, $\GG$ its Borel
$\sigma$-algebra, $\mu$ a Borel probability measure on $G$ and $(B,
\BB, \beta, T)$ the one-sided Bernoulli shift on the alphabet $(G,
\GG, \mu)$. 

Let $(X, \XX)$ be a standard Borel space equipped with a Borel action
of $G$. Let $\nu$ be a Borel probability measure on $X$ which is
$\mu$-stationary, i.e. $\mu * \nu = \nu$. 

We denote by $T^X$ the transformation on $B^X = B \times X$ given by,
for $(b,x) \in B^X$, 
\eq{eq: 3.1}{
T^X(b,x) = (Tb, b_0^{-1}x).
}
We denote, for $n \geq 0$, by $\BB_n$ the sub-$\sigma$-algebra of
$\BB$ generated by the coordinate functions $b_i, \ i=0,1,\ldots, n$,
and denote by $\pi: B^X \to B$ the projection onto the first factor. 

\begin{lem}\name{lem: 3.1}
Let $\nu$ be a $\mu$-stationary probability measure on $X$. 
\begin{itemize}
\item[a)] 
There is a unique probability measure $\beta^X$ on $(B^X, \BB \otimes
\XX)$ such that, for any $n \geq 0$ and any  $\BB_n \otimes \XX$-measurable bounded
function $\varphi$, 
\eq{eq: 3.2}{
\int_{B^X} \varphi(b,x) d\beta^X(b,x) = \int_{B^X} \varphi(b,b_0
\cdots b_{n-1}y) d\beta(b)d\nu(y). 
} 
\item[b)]
The probability measure $\beta^X$ is $T^X$-invariant and satisfies
$\pi_*\beta^X = \beta$. 
\end{itemize}
\end{lem}

\begin{proof}
a). For $n\geq 0$ we introduce the probability measure on $\BB_n
\otimes X$ defined by $\beta^X_n = \int_B \delta_b \otimes (b_0 \dots
b_{n-1})_*\nu d\beta(b).$ Since $\nu$ is $\mu$-stationary, for every
$n \geq 0$, the measure $\beta_{n+1}^X$ coincides with $\beta_n^X$ on
the $\sigma$-algebra $\BB_n \otimes \XX$. By the theorem of
Caratheodory, it follows that there is a unique probability measure
$\beta^X$ on $\BB\otimes \XX$ which coincides with $\beta^X_n$ on
$\BB_n \otimes \XX$ for every $n \geq 0$. 

b). For any $n \geq 0$, one has $(T^X)^{-1} (\BB_n \otimes \XX)
\subset (\BB_{n+1} \otimes \XX)$ and,
for any bounded $\BB_n \otimes \XX$-measurable function $\varphi$, by
definition, 
\[
\begin{split}
\int_{B^X} \varphi(T^X(b,x)) d\beta_{n+1}^X(b,x) & = \int_{B^X}
\varphi (Tb, b_0^{-1}b_0b_1\cdots b_ny) d\beta(b) d\nu(y) \\
& = \int_{B^X} \varphi(b,x) d\beta_n^X(b,x).
\end{split}
\]
It follows that $T^X_*\beta^X = \beta^X$. In addition, equation
\equ{eq: 3.2} with $n=0$ gives the equality $\pi_*\beta^X = \beta$. 
\end{proof}

We denote by $\BB^X$ the completion of the $\sigma$-algebra $\BB
\otimes \XX$ with respect to the measure $\beta^X$. 

\subsection{Martingales and conditional probabilities}\name{subsec:
  3.2} 
{\em In this section, we associate with each stationary probability
  measure $\nu$ on $X$ a measurable and $T$-equivariant family
  $(\nu_b)_{b \in B}$ of probability measures on $X$. } 

The disintegration of $\beta^X$ along the factor map $\pi$, proves the
existence of a $\BB$-measurable map $B \to \PP(X), b \mapsto \nu_b$,
such that 
\eq{eq: 3.3}{
\beta^X = \int_B \delta_b \otimes \nu_b d\beta(b).
}
In other words, for any bounded $\BB^X$-measurable function $\varphi$
on $B^X$, one has 
\eq{eq: 3.4}{
\beta^X(\varphi) = \int_B \int_X \varphi(b,y) d\nu_b(y) d\beta(b).
}
Also one has the following equality for $\beta^X$-a.e. $(b,x) \in B^X$
\eq{eq: 3.5}{
\EE\left( \varphi | \pi^{-1}\BB\right) (b,x) = \int_X\varphi(b,y) d\nu_b(y),
}
where the conditional expectation is taken relative to the probability
measure $\beta^X$. 

The following lemma interprets the conditional probabilities $\nu_b$
as limit probabilities. 
\begin{lem}\name{lem: 3.2}
Let $\nu$ be a $\mu$-stationary probability measure on $X$ and let $b
\mapsto \nu_b$ be the $\BB$-measurable family of probability measures
on $X$ constructed above. 
\begin{itemize}
\item[a)]
For any bounded Borel function $f$ on $X$, for $\beta$-a.e. $b \in B$,
we have 
\eq{eq: 3.6}{
\nu_b(f) = \lim_{p \to \infty} (b_{0*} \cdots b_{p*}\nu)(f).
}
\item[b)]
For $\beta$-a.e. $b \in B$, we have 
\eq{eq: 3.7}{
\nu_b = b_{0*} \nu_{Tb}.
}
\item[c)]
We have 
\eq{eq: 3.8}{
\nu = \int_B \nu_b d\beta(b).
}

\item[d)]
The map $b \mapsto \nu_b$ is the unique $\BB$-measurable map $B \to
\PP(X)$ for which \equ{eq: 3.7} and \equ{eq: 3.8} hold. 
\item[e)]
Conversely, for any $\BB$-measurable family $b \mapsto \nu_b \in
\PP(X)$ satisfying \equ{eq: 3.7}, the measure $\nu$ given by \equ{eq:
  3.8} is $\mu$-stationary. 
\end{itemize}
\end{lem}

\begin{proof}
a). For $\beta$-a.e. $b \in B$, we denote by $\nu_{b,p}$ the
probability measure $\nu_{b,p} = b_{0*} \cdots b_{p*} \nu \in
\PP(X)$. The proof is based on an explicit formula for the conditional
expectation: for each $p \geq 0$, for any bounded $\XX$-measurable
function $f$, which we will consider as a function on $B^X$, for
$\beta^X$-a.e. $(b,x) \in B^X$, one has
\eq{eq: 3.9}{
\EE \left(f | \pi^{-1} \BB_p \right)(b,x) = \int_X f(b_0 \cdots
b_{p-1}x') d\nu(x').
}
In fact, the right hand side of this equation is $\pi^{-1}
\BB_p$-measurable, and for each $\pi^{-1}\BB_p$-measurable function
$\psi$, one has by \equ{eq: 3.2}, 
$$
\int_{B^X} f \psi d\beta^X = \int_B \psi (b_0, \ldots, b_{p-1})\int_X
f(b_0 \cdots b_{p-1}x') d\nu(x') d\beta(b),
$$
and \equ{eq: 3.9} follows. The result is thus an immediate consequence
of the Martingale convergence theorem, since, by definition, for
$\beta$-a.e. $b \in B$, $\nu_b(f) = \EE \left(f | \pi^{-1} \BB \right)
(b).$

b) This equality follows from a) applied to a countable collection of
functions $f$ which generate the Borel $\sigma$-algebra $\XX$. 

c) It follows from \equ{eq: 3.2} and \equ{eq: 3.3} that for any
bounded Borel function $f$ on $X$, one has $\nu(f) = \int_{B^X} f(x)
d\beta^X(b,x) = \int_B \nu_b(f) d\beta(b).$ 

d) Let $b \mapsto \nu'_b$ be a $\BB$-measurable collection of
probability measures on $X$ satisfying the conditions. We will define
the probability measure $\lambda = \int_B \delta_b \otimes \nu'_b
d\beta(b)$ on $B^X$ and prove that $\lambda = \beta^X$. To this end,
we compute, for any positive $\BB_n \otimes \XX$-measurable function
$\varphi$ on $B^X$, using the two properties \equ{eq: 3.7} and
\equ{eq: 3.8} for the family $\nu'_b$ and using \equ{eq: 3.2}, 
\[
\begin{split}
\lambda(\varphi) & = \int_B \int_X \varphi(b,x) d\nu'_b(x) d\beta(b)
\\
& = \int_B \int_B \int_X \varphi (b_0 \cdots b_{n-1} b', b_0 \cdots
b_{n-1}y) d\nu'_{b'} (y) d\beta(b') d\beta(b) \\
& = 
\int_B \int_X \varphi(b, b_0 \cdots b_{n-1} y) d\beta(b) d\nu(y) = \beta^X(\varphi).
\end{split}
\]
This implies $\lambda = \beta^X$ since, by the uniqueness of
disintegration, for $\beta$-a.e. $b$, one has $\nu'_b = \nu_b$. 

e) One has 
$$
\mu * \nu = \int_G \int_B g_* \nu_b d\beta(b) d\mu(g) = \int_B b_{0*}
\nu_{Tb} d\beta(b) = \int_B \nu_b d\beta(b) = \nu.
$$
\end{proof} 

\begin{remark}
Whenever $X$ is a metrizable separable locally compact space and the
action of $G$ on $X$ is continuous (this will always be the case in
our applications), one then has
\eq{eq: 3.10}{
\nu_b = \lim_{p \to \infty} b_{0*} \cdots b_{p*} \nu.
}
This is the original introduction of the object by Furstenberg
\cite{ref 13}.
\end{remark}

\begin{remark}One  easily shows that the probability measure $\nu$ is
  $\mu$-ergodic if and only if the probability measure $\beta^X$ is
  $T^X$-ergodic.  
\end{remark}

We indicate a nice application of these constructions. 

\begin{cor}\name{cor: 3.5}
Let $\mu$ be a probability measure on $G$, let $\nu$ and $\nu'$ be two
$\mu$-stationary measures on two standard Borel spaces $(X, \XX)$ and
$(X', \XX')$, endowed with a Borel action of $G$. Then, the
probability measure $\nu'' = \int_B \nu_b \otimes \nu'_b d\beta(b)$ is
a $\mu$-stationary Borel probability measure on the product space $X
\times X'$. 

\end{cor}

\begin{proof}
In fact, the $\BB$-measurable family $b \mapsto \nu''_b = \nu_b
\otimes \nu'_b$ of probability measures on $X \times X'$ satisfies,
for $\beta$-a.e. $b \in B$, the equality $b_{0*}\nu''_{Tb} =
\nu''_b$. 
\end{proof}

\subsection{Fibered systems over a suspension}\name{subsec: 3.3}
{\em The dynamical system which we will need for our problem is a
  fibered product over a suspension.}

Let $M$ be a compact metrizable topological group and let $\tau =
(\tau_{\R}, \tau_M) : B \times \R_+ \times M$ a $\BB$-measurable map
with $\tau_\R \neq 0$. We denote by $(B^\tau, \BB^\tau, \beta^\tau,
T^\tau)$ the semi-flow obtained by the suspension of $(B, \BB, \beta,
T)$ using $\tau$, defined in \S \ref{subsec: 2.2}. We will now construct
a fibered semi-flow $B^{\tau, X}$ over $B^\tau$. 

For $\ell \geq 0$ and for $\beta^\tau$-a.e. $c = (b,k,m) \in B^\tau$,
we introduce the map $\rho_\ell(c)$ of $X$ given by, for any $x \in
X$, 
$$
\rho_\ell(c)x = b^{-1}_{p_\ell(b,k)-1} \cdots b_0^{-1} x,
$$
and denote $\nu_c = \nu_b$. We then have the following equivariance
property for the probability measures on $X$:

\begin{lem}\name{lem: 3.6}
For $\beta^\tau$-a.e. $c = (b,k,m) \in B^\tau$ and for every $\ell
\geq 0$, one has 
$$
\nu_{T^\tau_\ell c} = \rho_\ell(c)_* \nu_c.
$$
\end{lem}
\begin{proof}
Because of Lemma \ref{lem: 3.2}(b) and the equality $\nu_{T^\tau_\ell
  c } = \nu_{T^{p_\ell}(b,k) b}$, we have also $\nu_c = (b_0 \cdots
b_{p_\ell(b,k)-1})_* \nu_{T^\tau_\ell c}.$ 
\end{proof}

We define the semi-flow $\left( B^{\tau, X}, \BB^{\tau, X},
  \beta^{\tau, X}, T^{\tau, X}\right)$ fibered over $\left( B^\tau,
  \BB^\tau, \beta^\tau, T^\tau\right)$ as follows. We set $B^{\tau, X}
= B^\tau \times X$ and 
$$
\beta^{\tau, X} = \int_{B^\tau} \delta_c \otimes \nu_c d\beta^\tau(c).
$$
We denote by $\BB^{\tau, X}$ the completion of the product
$\sigma$-algebra $\BB^\tau \otimes \XX$ with respect to the
probability measure $\beta^{\tau, X}$ and, for $(c,x) \in B^{\tau, X}$
and $\ell \geq 0$, we set 
$$
T^{\tau, X}_\ell (c,x) = \left( T^\tau_\ell c, \rho_\ell(c)x\right).
$$

\begin{lem}\name{lem: 3.7}
For all $\ell \geq 0$, the transformation $T_\ell^{\tau, X}$ of
$B^{\tau, X}$ preserves the measure $\beta^{\tau, X}$. 
\end{lem}

\begin{proof}
This follows from Lemmas \ref{lem: 2.4} and \ref{lem: 3.6}. 
\end{proof}

Denote $\QQ_\ell^{\tau, X} = \left(T_\ell^{\tau, X} \right)^{-1}
(\BB^{\tau, X})$ and denote by $\QQ_\infty^{\tau, X}$ the tail
$\sigma$-algebra of $\left(B^{\tau, X}, \BB^{\tau, X}, \beta^{\tau,
    X}, T^{\tau, X} \right)$, that is the decreasing intersection of
sub-$\sigma$-algebras $\QQ_\infty^{\tau, X} = \bigcap_{\ell \geq 0}
\QQ_\ell^{\tau, X}$. Similarly, denote by $\QQ_\ell$ the decreasing
family of $\sigma$-algebras $\QQ_\ell = (T^\tau_\ell)^{-1} (\BB^\tau)$
and by $c \mapsto \beta^\ell_c$ the conditional measure of
$\beta^\tau$ relative to $\QQ_\ell$. 

We can conclude the preceding discussion with the following corollary
which is at the heart of our drift argument. 

\begin{cor}\name{cor: 3.8}
For any $\beta^{\tau, X}$-integrable function $\varphi : B^{\tau, X}
\to \R$, for every $\ell \geq 0$, for $\beta^{\tau, X}$-a.e. $(c,x)
\in B^{\tau, X}$, one has 
\eq{eq: 3.11}{
\EE \left( \varphi | \QQ^{\tau, X}_\ell \right) (c,x) = \int_{B^\tau}
\varphi(c', \rho_\ell(c')^{-1} \rho_\ell(c)x) d\beta^\ell_c(c').
}
\end{cor}
\begin{proof}
This follows from Lemma \ref{lem: 2.5}.
\end{proof}

\subsection{Measure of relative stable leaves}\name{subsec: 3.4}
{\em In order to be able to apply our drift argument, we will need to
  know that the probability measures $\nu_b$ give no mass to the
  relative stable leaves of the factor map $B^{\tau, X} \to
  B^\tau$. Proposition \ref{prop: 3.9} below will give us a useful
  criterion which will enable us to prove this. }

We will assume from now on that $X$ is a locally compact metrizable
topological space and that the action of $G$ on $X$ is continuous. We
denote by $d$ a metric on $X$ inducing the topology. For $(b,x)$ in $B
\times X$, we denote by 
$$
W_b(x) = \{x' \in X: \lim_{p \to \infty} d(\rho_p(b)x, \rho_p(b)x') =0\}
$$
the stable leaf relative to $(b,x)$. This leaf does not depend on the
choice of the metric $d$ whenever $X$ is compact, but may depend on
$d$ in general. However, in all cases, one has the following
proposition. Recall that a continuous map is called {\em proper} if
the inverse image of any compact subset is compact. Denote by $A_\mu$
the averaging operator on $X \times X$ given by, for any positive
function $v$ on $X \times X$ and any $(x,y)$ in $X \times X$, 
$$
A_\mu(v)(x,y) = \int_G v(gx,gy) d\mu(g).
$$
This operator is thus the convolution operator of the image $\check
\mu$ of the measure $\mu$ under inversion $g \mapsto g^{-1}$. We
denote by $\Delta_X$ the diagonal in $X \times X$. 

\begin{prop}\name{prop: 3.9}
Suppose the following hypothesis (HC): 

There exists a function $v: (X \times X) \sm \Delta \to [0, \infty)$
such that, for any compact subset $K \subset X$, the restriction of
$v$ to $K \times K \sm \Delta$ is proper and there are constants $a
\in (0,1) $ and $C>0$ such that $A_\mu(v) \leq av +C.$

Let $\nu$ be a $\mu$-stationary non-atomic Borel probability measure
on $X$. Then for $\beta^X$-a.e. $(b,x) \in B \times X$, one has 
$$
\nu_b(W_b(x)) =0.
$$
\end{prop}
Hypothesis (HC) signifies that on average, $\mu$ contracts the
function $v$ at a fixed rate.

The proof of this fact follows three steps. The first step is the most
delicate, and is contained in the following lemma. 

\begin{lem}\name{lem: 3.10}
Assume hypothesis (HC), and let $\nu$ be a $\mu$-stationary Borel
probability measure such that, for $\beta$-a.e. $b \in B$, the
probability measure $\nu_b$ is a Dirac mass. Then $\nu$ is a Dirac
mass. 
\end{lem}

\begin{proof}
Let $\kappa: B \to X$ denote the $\BB$-measurable map such that, for
$\beta$-a.e. $b \in B$, one has
\eq{eq: 3.12}{
\nu_b = \delta_{\kappa(b)}.
}
The strategy will consist of studying the
corresponding random walk on $X \times X$. Roughly speaking, the
existence of $\kappa$ 
and the Chacon-Ornstein ergodic theorem will ensure that this random
walk approaches the diagonal $\Delta_X$ while the existence of $v$
pushes the random walk away from the diagonal. Here are the details. 

For $g \in G$ and $b = (b_0, b_1, \ldots) \in B$, let $gb = (g,b_0,
b_1, \ldots)$. By Lemma \ref{lem: 3.2}(b) we have, for $\mu$-a.e. $g
\in G$ and $\beta$-a.e. $b \in B$, 
$$
\kappa (gb) = g\kappa(b).
$$
By Lemma \ref{lem: 3.2}(c), we also have the equality 
$$
\nu = \kappa_*\beta.
$$
Endow $B = G^\N$ with the product topology. By Lusin's theorem, for
every $\vre>0$, there is a compact subset $K_0 \subset B$ such that
$\beta(K_0) = 1-\vre$ and the restriction of $\kappa$ to $K_0$ is
uniformly continuous. Denote by $K$ the compact image $K =
\kappa(K_0)$. Since the restriction of $v$ to $K \times K \sm \Delta$
is proper, one has 
\begin{equation}\label{eq: 3.13}
\begin{split}
\forall M>0, \exists n_M > 0, &  \forall n \geq n_M, \forall b,b' \in
B, \\ 
\forall g_1, \ldots, g_n \in G  \text{ such that } &   g_1 \cdots g_nb
\in K_0  \text{ and } g_1 \cdots g_n b' \in K_0, \\  \text{ we have } 
 v(\kappa(g_1 \cdots g_nb), & \kappa (g_1 \cdots g_n b')) \geq M.
\end{split}
\end{equation}
We now introduce the transfer operator $L_\mu$ on $B$ given by, for
each $\varphi_0 \in L^1(B, \beta)$, for $\beta$-a.e. $b \in B$, 
$$
(L_\mu\varphi_0)(b) = \int_G \varphi_0(gb) d\mu(g).
$$
Since it is the adjoint of the shift $T$, $L_\mu$ is an ergodic
operator. The theorem of Chacon-Ornstein \cite{ref 5}, applied to the
function $\varphi_0 = 1_{K_0}$, ensures that for $b$ outside a subset
$N \subset B$ of zero measure, we have the equality
\eq{eq: 3.14}{
\lim_{p \to \infty} \frac{1}{p} \sum_{0 \leq n \leq p} (L^n_\mu
1_{K_0})(b) = \beta(K_0) = 1 - \vre.
}
By possibly  increasing the set $N$, we may also assume that for any
$b \in B \sm N$, for any integer $n \geq 0$, and for $\mu^{\otimes
  n}$-a.e. $(g_1, \ldots, g_n) \in G^n$, one has $\kappa(g_1 \cdots
g_nb) = g_1 \cdots g_n \kappa(b)$. 

Suppose by contradiction that $\nu$ is not a Dirac mass. Then the set 
$$
E = \{(b,b') \in B \times B : \kappa (b) \neq \kappa (b') \}
$$
is of positive measure with respect to $\beta \otimes \beta$. Therefore we
can find points $b_0$ and $b'_0$ outside of $N$ such that 
\eq{eq: 3.15}{
\kappa(b_0) \neq \kappa(b'_0).
}
We now use condition (HC). It implies that for all $n \geq 0$, one has 
$$
A^n_\mu v \leq a^n v + (1+ \cdots + a^{n-1})C.
$$
For every $x \neq x' \in X$, we deduce the upper bound 
\eq{eq: 3.16}{
\frac{1}{p} \sum_{0 \leq n \leq p} (A^n_\mu v) (x,x') \leq
\frac{1}{p(1-a)} v(x,x') + \frac{1}{1-a}C.
}
We will now apply this upper bound to the points $x = \kappa (b_0)$
and $x' = \kappa(b'_0)$. Fix $M>0$. Note that, thanks to \equ{eq:
  3.14}, there exists an integer $p_0 \geq n_M$ such that for all $p
\geq p_0$, 
$$
\frac{1}{p} \sum_{0 \leq n\leq p} (L^n_\mu 1_{K_0})(b_0) \geq 1-2\vre
\text{ and } \frac{1}{p} \sum_{0 \leq n \leq p} (L^n_\mu
1_{K_0})(b_0') \geq 1-2\vre.
$$
As a consequence, 
$$
\frac{1}{p} \sum_{0 \leq n \leq p} (A^n_\mu v)(\kappa(b_0),
\kappa(b'_0)) \geq \left (1 - 4 \vre - \frac{p_0}{p}\right ) M .
$$
Taking a limit as $p \to \infty$ and using \equ{eq: 3.16} we obtain
$$
(1-4\vre)M \leq C/(1-a).
$$
Since $M$ was arbitrary, we get a contradiction as soon as $\vre<
1/4$. Therefore $\nu$ is a Dirac mass. 
\end{proof}

The second step is the following lemma:

\begin{lem}\name{lem: 3.11}
Under assumption (HC), let $\nu$ be a non-atomic $\mu$-stationary probability
measure on $X$. Then for $\beta$-a.e. $b \in B$, the probability
measure $\nu_b$ is non-atomic. 

\end{lem}

\begin{proof}
The strategy consists, after several reductions involving the
ergodicity of $\beta$, in constructing a stationary probability measure
on a space $Y$ on which one can apply Lemma \ref{lem: 3.10}. 

Suppose by contradiction that the set $D = \{b \in B: \nu_b \text{ has
  atoms} \}$ is of positive measure. Since $\nu_b = b_{0*}\nu_{Tb}$,
the set $D$ is $T$-invariant. Since $\beta$ is $T$-ergodic, this means
that $\beta(D)=1$. The same argument also shows that the maximal mass
$M_b$ of an atom of $\nu_b$ is a $\beta$-almost surely constant
function and that the number $N_b$ of atoms whose $\nu_b$ measure is
$m_b$ is also a.e. constant. We denote this mass by $m_0$ and this
number of atoms by $N_0$. Denote by $\nu'_b$ the probability measure
with $N_0$ atoms of $\nu_b$ each of mass $m_0$. We also have the
equality $\nu'_b = b_{0*} \nu'_{Tb}$. By Lemma \ref{lem: 3.2}(e), the
probability measure $\nu' = \int_B \nu'_b d\beta(b)$ on $X$ is also
$\mu$-stationary and one can write $\nu$ as the sum of $m_0 \nu'$ and
a stationary measure of mass $(1-m_0)$. By assumption, $\nu'$ is also
non-atomic, and by Lemma \ref{lem: 3.2}(d), the measures $\nu'_b$ are
the limit measures of $\nu'$, and thus we can henceforth assume that
$\nu = \nu'$. 

Let $S_{N_0}$ denote the group of permutations of $\{1, \ldots, N_0\}$
and let $Y$ denote the quotient $X^{N_0}/S_{N_0}$ and $p: X^{N_0} \to
Y$ the projection. The group $G$ acts naturally on $Y$. We check that
$Y$ satisfies hypothesis (HC). Let $v$ denote the function and let
$a,C$ denote the constants which appear in hypothesis (HC) for $X$ and
introduce the map $w: Y \times Y \sm \Delta \to [0, \infty)$ given,
for $y = p(x_1, \ldots, x_{N_0})$ and $y' = p(x'_1, \ldots, x'_{N_0})$
with $x_i, x'_i \in X$, by
$$
w(y,y') = \sum_{\sigma \in S_{N_0}} \min_{1 \leq i \leq N_0} v(x_i, x'_{\sigma(i)}).
$$
This map $w$ is certainly continuous and proper on $K \times K \sm
\Delta$ for any compact subset $K \subset Y$. It also satisfies an
upper bound 
$$
A_\mu(w) \leq aw + CN_0!.
$$
Introduce the family $b \mapsto \nu''_b = p_*(\nu_b^{\otimes N_0})$ of
probability measures on $Y$. We also have the equality $\nu''_b =
b_{0*} \nu''_{Tb}$. By Lemma \ref{lem: 3.2}(e), the probability
measure $\nu'' = \int_B p_*(\nu_b^{\otimes N_0}) d\beta(b)$ is
$\mu$-stationary. By construction, for $\beta$-a.e. $b \in B$ the
measure $\nu''_b$ is a Dirac mass. Lemma \ref{lem: 3.10} then shows
that $\nu''$ is also a Dirac mass $\delta_{y_0}$. Therefore, for
$\beta$-a.e. $b \in B$, $\nu''_b = \delta_{y_0}$ and hence $\nu$ is of
finite support, a contradiction. 
\end{proof}

The last step does not use assumption (HC).

\begin{lem}\name{lem: 3.12}
Let $\nu$ be a $\mu$-stationary probability measure on $X$ such that,
for $\beta$-a.e. $b \in B$, the measure $\nu_b$ is non-atomic. Then
for $\beta^X$-a.e. $(b,x) \in B \times X$, $\nu_b(W_b(x)) =0$. 
\end{lem}

\begin{proof}
Consider the transformation on $B \times X \times X$ given by, for
$(b,x,x') \in B \times X \times X$, 
$$
R(b,x,x') = (Tb, b_0^{-1}x, b_0^{-1}x').
$$
Lemma \ref{lem: 3.1} and Corollary \ref{cor: 3.5} show that $R$
preserves the probability measure 
$$
\Lambda = \int_B \delta_b \otimes \nu_b \otimes \nu_b \, d\beta(b).
$$
Denote 
$$
Z = \{(b,x,x') \in B \times X \times X : \lim_{p \to \infty}
d(\rho_p(b)x, \rho_p(b)x') =0\}
$$
and, for $(b,x,x') \in B \times X \times X$, write $\varphi(b,x,x') =
d(x,x')$. By assumption, for $\beta$-a.e. $b$, the measure $\nu_b$ is
non-atomic, and hence $\nu_b \otimes \nu_b$ gives no mass to the
diagonal $X \times X$. Therefore the function $\varphi$ is
$\Lambda$-a.e. nonzero. By construction, for $\Lambda$-a.e. $z \in Z$,
one has $\lim_{p \to \infty} \varphi(R^p(z)) =0$ and thus, by the
Poincar\'e recurrence theorem, $\Lambda(Z)=0$, as required. 
\end{proof}

\begin{proof}[Proof of Proposition \ref{prop: 3.9}]
Follows from Lemma \ref{lem: 3.11} and \ref{lem: 3.12}. 
\end{proof}

\section{Conditional measures}
{\em In this chapter we collect certain properties of conditional
  measures of a probability measure for a Borel action of a locally
  compact group.}

\subsection{Conditional measures}\name{subsec: 4.1}
{\em We recall the construction of conditional
  measures. \combarak{Often called leaf-wise measures.}}

Let $R$ be a locally compact separable metrizable group and $(Z, \ZZ)$
a standard Borel space with a Borel action of $R$. Let $\lambda$ be a
Borel probability measure on $Z$. Suppose that the stabilizer
subgroups for the action of $R$ on $Z$ are discrete. We will now
explain how the action of $R$ on $Z$ makes it possible to
`disintegrate the measure $\lambda$ along $R$-orbits', to obtain
measures on $R$ which are unique up to normalization. More precisely: 

Let $\MM(R)$ denote the space of positive nonzero Radon measures on
$R$ and let $\MM_1(R) = \MM(R) /\simeq$ be the space of such measures
up to scaling: two Radon measures $\sigma_1, \sigma_2$ are called
equal up to scaling, and we write
$
\sigma_1 \simeq \sigma_2$, if there is $ c>0$ such that 
$\sigma_2 = c\sigma_1.$
We can choose a representative of each equivalence class: we fix an
increasing 
sequence of compact subsets $(K_n)$ of $R$ which cover $R$ and choose
$\sigma$ so that $\sigma(K_n) =1$, where $n$ is the smallest $m$ for
which $\sigma(K_m)>0$. 

We say that a Borel subset $\Sigma \subset Z$ is a {\em discrete
  section} of the action of $R$ if, for any $z \in Z$, the set of
visit times $\{r \in R: rz \in \Sigma\}$ is discrete and closed in
$R$. The main theorem of \cite{ref 21} shows that there is a discrete
section $\Sigma$ for the action of $R$ such that $R\Sigma = Z$. 

We choose a discrete section $\Sigma$ for the action of $Z$ on $R$ and
denote $a: R \times \Sigma \to Z, (r,z) \mapsto rz$. The measure
$a^*\lambda$ on $R \times \Sigma$ defined, for any positive Borel
function $f$ on $R \times \Sigma$, by 
\eq{eq: 4.1}{
a^*\lambda(f) = \int_Z \left(\sum_{(r,z') \in a^{-1}(z)} f(r,z')
\right) d\lambda(z),
}
is a $\sigma$-finite Borel measure on $R \times \Sigma$. This follows
from the fact that for any compact subset $C \subset R$, and any $z
\in Z$, the set $(C \times \Sigma) \cap a^{-1}(z)$ is finite. 

We denote $\pi_\Sigma : R\times \Sigma \to \Sigma$ the projection on
the second factor, and by $\lambda_\Sigma$ the image under
$\pi_\Sigma$ of a finite measure on $R \times \Sigma$ equivalent to
$a^*\lambda$. We therefore have, for any positive Borel function on $R
\times Z$, 
\eq{eq: 4.2}{
a^*\lambda(f) = \int_\Sigma \int_R f(r,z) d\sigma_\Sigma(z)(r)
d\lambda_\Sigma(z). 
}
Note that the conditional measures $\sigma_\Sigma(z)$ are also Radon
measures on $R$. This results once more from the finiteness of the
sets $(C \times \Sigma) \cap a^{-1}(z)$. 

We denote by $t_r$ the right-translation by an element $r \in R$. 

\begin{lem}\name{lem: 4.1}
Let $\Sigma$ be a discrete section for the action of $R$ on $Z$. For
$\lambda_\Sigma$-a.e. $z \in \Sigma$, for all $r \in R$ such that $rz
\in \Sigma$, we have
$$
\sigma_\Sigma(z) \simeq t_{r*} \sigma_\Sigma(rz).
$$
\end{lem}
\begin{proof}
The difficulty comes from the fact that one wants this condition to be
satisfied for an uncountable family of elements $r \in R$. To deal
with this difficulty, it suffices to remark that there is a countable
family, indexed by $i \in \N$, of Borel sets $\Sigma_i \subset
\Sigma$, and Borel maps $r_i: \Sigma_i \to R$, such that 
$$
\{(z,r) \in \Sigma \times R : r z \in \Sigma\} = \bigcup_{i \in \N}
\{(z, r_i(z)) : z \in \Sigma_i \},
$$
and such that, for $\lambda_\Sigma$-a.e. $z \in \Sigma_i, \
\sigma_\Sigma(z) \simeq t_{r_i(z)*} \sigma_\Sigma(r_i(z)\,z).$ 
\end{proof}

\begin{prop}\name{prop: 4.2}
Consider a Borel action with discrete stabilizers of a locally compact
separable metrizable group $R$ on a standard Borel space $(Z, \ZZ)$. 

Then there is a Borel map $\sigma:  Z \to \MM_1(R)$ and a Borel
subset $E \subset Z$ such that $\lambda(Z \sm E) =0$ and such that,
for any discrete section $\Sigma \subset Z$ for the action of $R$, for
$\lambda_\Sigma$-a.e. $z_0 \in \Sigma$, for every $r \in R$ such that
$rz_0 \in E$, 
$$
\sigma(z_0) \simeq t_{r*} \sigma_\Sigma(rz_0).
$$
This map $\sigma$ is unique up to a set of $\lambda$-measure zero. 

For every $r \in R$ and every $z \in E$ such that $rz \in E$, we have 
\eq{eq: 4.3}{
\sigma(z) \simeq t_{r*}(\sigma(rz)).
}
\end{prop}
The measure $\sigma(z)$ is called the conditional measure of $z$ along
the action of $R$. 

\begin{proof}
We choose a discrete section $\Sigma_0$ such that $R\Sigma_0 = Z$. By
Lemma \ref{lem: 4.1}, for $\lambda$-a.e. $z \in Z$, if one writes $z =
rz_0$ with $r \in R$ and $z_0 \in \Sigma_0$, the measure $\sigma(z) =
t^{-1}_{r*} \sigma_{\Sigma_0}(z_0) \in \MM_1(R)$ does not depend on
choices, i.e. different choices of $z_0$ only affect it by rescaling. 

This defines the map $\sigma$. The asserted property of $\sigma$
follows from the Lemma applied to $\Sigma \cup \Sigma_0$ which is also
a discrete section. Assertion \equ{eq: 4.3} follows. Uniqueness of
$\sigma$ is clear.
\end{proof}

The use of conditional measures in geometric ergodic theory is based,
among others, on the work of Ledrappier-Young. Its use in problems
of measure classification on homogeneous spaces has appeared in \cite{
 ref 4} and earlier in work of Katok and Spatzier.

\subsection{Disintegration along stabilizers}
{\em In this section we explain how to exploit the invariance
  properties under translation, of conditional measures along an
  action.}

Denote by $\Gr(\R^d)$ the Grassmannian variety of $\R^d$.  
The following proposition asserts that the disintegration of $\lambda$
to conditional measures  along the stabilizer gives probability
measures invariant under the stabilizer. In a topological group $S$,
we denote by $S_0$ the connected component of the identity. 

\begin{prop}\name{prop: 4.3}
Let $(Z, \ZZ)$ be a standard Borel space endowed with a Borel action
of $\R^d$ with discrete stabilizers, and let $\lambda$ be a Borel
probability measure on $Z$. For $\lambda$-a.e. $z \in Z$, we denote by
$\sigma(z)$ the conditional measure of $z$ for the action of $\R^d$,
and 
$$
V_z = \{r \in \R^d: t_{r*}\sigma(z) = \sigma(z)\}_0,
$$
and by
$$
\lambda = \int_Z \lambda_z d\lambda(z)
$$
the distintegration of $\lambda$ along the map $Z \to \Gr(\R^d), \, z
\mapsto V_z$. Then for $\lambda$-a.e. $z \in Z$, the probability
measure $\lambda_z$ is $V_z$-invariant. 
\end{prop}

This proposition is a consequence of the following three lemmas. The
first one uses notation which are different from those used in
Proposition \ref{prop: 4.3}. 

\begin{lem}\name{lem: 4.4}
Let $(Z, \ZZ, \lambda)$ be a Lebesgue space, $(Y, \YY)$ a standard
Borel space equipped with a Borel action of $\R^d$, $f:Z \to Y$ a
measurable map and $I: Z \to \Gr(\R^d)$ a measurable map such that for
$\lambda$-a.e. $z \in Z$, $I(z)$ stabilizes $f(z)$. 

Denote by $\lambda = \int_Z \lambda_z d\lambda(z)$ the disintegration
of $\lambda$ along $I$. 

Then for $\lambda$-a.e. $z \in Z$, for $\lambda_z$-a.e. $z' \in Z$,
the element $f(z')$ is $I(z)$-invariant. 
\end{lem}

\begin{proof}
In fact, for $\lambda$-a.e. $z \in Z$, for $\lambda_z$-a.e. $z' \in
Z$, we have from the definition of conditional measures, that
$I(z)=I(z')$ and hence, by assumption, $f(z')$ is $I(z')$-invariant. 
\end{proof}

The second lemma uses once more the notation of Proposition \ref{prop:
  4.3}. 

\begin{lem}\name{lem: 4.5}
Let $(Z, \ZZ)$ be a standard Borel space equipped with a Borel action
of $\R^d$ with discrete stabilizers, and let $\lambda$ be a Borel
probability measure on $Z$. Let $(Y_0, \YY_0)$ be a standard Borel
space and let $\varphi: Z \to Y_0$ be a measurable map for which there
exists a subset $E \subset Z$ such that $\lambda(Z \sm E)=0$ and for
every $z \in E$ and $r \in \R^d$ with $rz \in E$, we have $\varphi(z)
= \varphi(rz)$. Denote $z \mapsto \sigma(z) \in \MM(\R^d)$ the
conditional measure at $z$ of $\lambda$ along the $\R^d$-orbits, and
denote $\lambda = \int_Z \lambda_z d\lambda(z)$ the disintegration of
$\lambda$ along $\varphi$. Then, for
$\lambda$-a.e. $z \in Z$, for $\lambda_z$-a.e. $z' \in Z$,
$\sigma(z')$ is also the conditional measure of $z'$ of $\lambda_z$
along the action of $\R^d$. 
\end{lem}

\begin{proof}
We adapt the argument of transitivity of the disintegration of
measures in this context. 

Recall the gist of the argument in the classical context: we are given
a Lebesgue space $(A, \A, \alpha)$, and two standard Borel spaces
$(B, \BB), (C, \CC)$ along with measurable maps $f:A \to B$ and $g: B
\to C$. Then almost surely, the conditional measures of $\alpha$ along
$f$ coincide with the conditionals along $f$ of the conditionals of
$\alpha$ along $g \circ f$. More precisely, denote $\alpha = \int_A
\alpha_a d\alpha(a)$ and $\alpha = \int_A \alpha_{a'} d\beta_a(a')$,
the disintegrations of $\alpha$ respectively along $f$ and along $g
\circ f$. We then have, for $\alpha$-a.e. $a$, the equality $\beta_a =
\int_A \alpha_{a'} d\beta_a(a')$ which gives the disintegration of
$\beta_a$ along $f$. 
\end{proof}

\begin{lem}\name{lem: 4.6}
Let $(Z, \ZZ)$ be a standard Borel space, equipped with a Borel action
of $\R^d$ with discrete stabilizers, $W$ a linear subspace of $\R^d$,
$\lambda$ a probability measure on $(Z, \ZZ)$, and $z \mapsto
\sigma(z) \in \MM(\R^d)$ the conditional measures at $z$ of $\lambda$
along the action of $\R^d$. Suppose that for $\lambda$-a.e. $z \in Z$,
$\sigma(z)$ is invariant under translations by $W$. Then $\lambda$ is
also invariant under the action of $W$. 
\end{lem}

\begin{proof}
As in \S \ref{subsec: 4.1}, denote by $\Sigma$ a discrete section for
the action of $\R^d$ such that $\R^d\Sigma = Z$ and let $a$ be the map $a: \R^d
\times \Sigma \to Z, \, (r,z) \mapsto rz$. By assumption the measure
$a^*\lambda$ is $W$-invariant, and hence so is the measure $\lambda$. 
\end{proof}

\begin{proof}[Proof of Proposition \ref{prop: 4.3}]
Applying Lemma \ref{lem: 4.4} with $Y = \MM(\R^d)$, $f=\sigma$ and
$I(z) = V_z$, and then Lemma \ref{lem: 4.5} with $Y_0 = \Gr(\R^d)$ and
$\varphi(z)=V_z$. We find that, for $\lambda$-a.e. $z \in Z$, for
$\lambda_z$ a.e. $z'\in Z$, the conditional measure $\sigma_z(z')$  of
$\lambda_z$ for the action of $\R^d$ on $Z$ is $V_z$-invariant and
hence, by Lemma \ref{lem: 4.6}, that the measure $\lambda_z$ is
$V_z$-invariant.  
\end{proof}

\section{Random walks on Lie groups}\name{sec: 5}
{\em In this chapter we introduce, for a strongly irreducible random
  walk, a dynamical system $\left(B^\tau, \BB^\tau, \beta^\tau, T^\tau
  \right)$ which is a suspension of the Bernoulli system $\left( B,
    \BB, \beta, T\right)$. We then study the asymptotic behavior of
  the random walk in order to be able to control the drift in \S
  \ref{subsec: 7.1}. 
}

\subsection{Stationary measures on the flag variety}\name{subsec: 5.1}
Let $G$ be a real semisimple virtually connected Lie group, that is it
has a finite number of connected components. 

\begin{dfn}
We say that a Borel probability measure on $G$ is Zariski dense if the
semigroup $\Gamma_\mu$ generated by the support of $\mu$ has a Zariski
dense image in the adjoint group $\Ad(G) \subset \GL(\mathfrak{g})$. 
\end{dfn}

Let $\mu$ be a Zariski dense probability measure on $G$ with compact
support. We also denote by $(B, \BB, \beta, T)$ the two-sided
Bernoulli shift on the alphabet $(G, \GG, \mu)$, where $\GG$ denotes
the Borel $\sigma$-algebra of $G$. 

Let $P \subset G$ be a minimal parabolic subgroup. Write $P = ZU$,
where $U$ is the unipotent radical of $P$ and $Z$ is a maximal
reductive subgroup of $P$. Denote by $A$ the Cartan subgroup of $Z$
and by $A^+$ the Weyl chamber of $A$ associated with an order
corresponding to the choice of $P$. Choose a Cartan involution of $G$
which leaves $Z$ invariant and denote by $K$ the maximal compact
subgroup of $G$ consisting of points fixed by this Cartan involution. 

Let $V$ be a real representation of $G$ of dimension $d$ which is
strongly irreducible, that is, its restriction to the connected
component of the identity in $G$ is also irreducible. Fix once and for
all a $K$-invariant Euclidean norm $\| \cdot \|$ on $V$ such that the
elements of $A$ act on $V$ in a symmetric fashion. \combarak{Does this
  mean the inner comes from an invariant inner product for which $A$
  is self-adjoint?}

Denote by $\chi$ the largest weight for $A$ in $V$, let $V_0 = V_\chi$
be the corresponding weight space in $V$, so that $PV_0 \subset V_0$,
and let $d_0 = \dim V_0$. Denote by $V'_0$ the subspace of $V$ which
is the sum of the other weight-subspaces, so that $V = V_0 \oplus
V'_0$. 

The following proposition is essentially due to Furstenberg and Kesten
\cite{ref 15}. Denote by $\Gr_{d_0}(V)$ the Grassmannian variety of
$d_0$-planes in $V$. 

\begin{prop}\name{prop: 5.2}
There are $\BB$-measurable maps $B \to \Gr_{d_0} (V), \, b \mapsto
V_b$ and $B \to \Gr_{d-d_0} (V), \, b \mapsto V'_b$, such that:
\begin{itemize}
\item[a)]
For $\beta$-a.e. $b \in B$, any accumulation point $m$ of the sequence
$\left( \frac{b_0 \cdots b_n}{\|b_0 \cdots b_n \|} \right)_n$, has as  its 
image $\mathrm{Im}(m) = V_b$ and is an isometry on $\ker(m)^\perp$. 
\item[b)]
For $\beta$-a.e. $b \in B$, any accumulation point $m'$ of the sequence
$\left( \frac{b_n \cdots b_0}{\|b_n \cdots b_0\|} \right)_n$, has
$\ker(m') = V'_b$ and is an isometry on $\ker(m')^\perp$. 
\item[c)]
For any hyperplane $W \subset V$, we have $\beta(\{b \in B: V_b
\subset W\})=0$. 
\item[d)] For any nonzero $v \in V$, we have $\beta(\{b \in B: v \in
  V'_b\})=0$. 
\item[e)]
For any $W \in \Gr_{d_0}(V)$, we have $\beta(\{b \in B: W \cap V'_b
\neq 0\})=0$. 
\item[f)]
For $\beta$-a.e. $b \in B$, the limit $\lambda_1 = \lim_{n \to \infty}
\frac{1}{n} \log \| b_0 \cdots b_n \|$ exists and is positive. 

\end{itemize}
\end{prop}

\begin{proof}
For a), c), and f) see \cite{ref 13} and \cite{ref 2}. 
The fact that the accumulation points $m$ are of rank $d_0 = \dim V_0$
is due to Goldsheid and Margulis \cite{ref 16}. It can also be deduced
from the existence of loxodromic elements in $\Gamma_\mu$. The fact
that the restriction of $m$ to the orthocomplement of its kernel is a
similarity is valid for any matrix $\pi$ of rank $d_0$ in the closure
$\overline{\R_*G} \subset \End(V)$ \combarak{Why is this true?}. One easily verifies this 
assertion thanks to the Cartan decomposition $G = KA^+K$. 

Assertions b) and d) are deduced from assertions a) and c) by passing
to the dual representation. 

Assertion e) is deduced from d) by passing to an irreducible
sub-representation of the representation of $G$ on $\bigwedge^{d_0} V$
generated by the line of highest weight $\bigwedge^{d_0} V_0$. 
\end{proof}

When applying Proposition \ref{prop: 5.2}(a) to a suitable
representation, one shows that there is a unique $\BB$-measurable map
$\xi: B \to G/P$ such that, for $\beta$-a.e. $b \in B$, 
$$
\xi(b) = b_0 \xi(Tb).
$$
The image measure $\xi_* \beta$ is therefore the unique
$\mu$-stationary measure on $G/P$. 

\begin{remark}
Certainly the spaces $V_b$ and $V'_b$ of Proposition \ref{prop: 5.2}
depend on the boundary map $\xi$. For $b \in B$, we denote by $\check
b$ the element $\check b = (b_0^{-1}, b_1^{-1}, \ldots) $ of $B$. For
$\beta$-a.e. $b \in B$, we have $V_b = \xi(b) V_0$ and $V'_b =
\xi(\check{b})V'_0$. We also have $V_b = b_0V_{Tb}$ and $V'_b = b_0 
V'_{Tb}$. 
\end{remark}

\subsection{The dynamical system $B^\tau$.}\name{subsec: 5.2}
{\em We want to construct an $\R \times M$-suspension $(B^\tau,
  T^\tau)$ of the Bernoulli shift associated to $\mu$ which enables us
  to estimate the asymptotic behavior of the induced random walk in an
  irreducible representation of $G$. We initially construct a
  function $\theta: B \to Z.$}

Let $s: G/P \to G/U$ be a Borel section of the projection $G/U \to
G/P$. In practice, for constructing such a section, one can utilize
Iwasawa decomposition or Bruhat decomposition. An explicit formula for
$s$ is not very important for us, because our constructions will not
depend on the choice of the section $s$. However, for simplicity,
suppose that the section is constructed with the aid of Iwasawa
decomposition. More precisely, write $M = Z \cap K$. The Iwasawa
decomposition $G = KP$ makes it possible to choose a section $s$ such
that, for every $k \in K$, 
\eq{eq: 5.1}{
s(kP) = km(k)U \text{ with } m(k) \in M.
}
We will say from now on that the function $s$ has {\em values in $K
  \mod M$}. The group $Z$ acts by right multiplication on $G/U$. 

We denote by $\sigma: G \times G/P \to Z$ the Borel cocycle given by,
for every $g \in G$ and $x \in G/P$, 
$$
gs(x) = s(gx)\sigma(g,x).
$$
We denote by $\theta: B \to Z$ the $\BB$-measurable map given by, for
$\beta$-a.e. $b \in B$, 
$$
\theta(b) = \sigma(b_0, \xi(Tb)).
$$
We introduce the bounded function $\theta_\R: B \to \R$ given, for
$\beta$-a.e. $b \in B$, by 
\eq{eq: 5.2}{
\theta_\R(b) = \log |\chi(\theta(b))|.
}
We will use the Furstenberg formula for the first Lyapunov exponent
\eq{eq: 5.3}{
\lambda_1 = \int_B \theta_\R(b) d\beta(b)
}
(see \cite{ref 13}, see also \cite[Thm. 1.8]{ref 12}), and the
positivity of the first Lyapunov exponent (Proposition \ref{prop:
  5.2}(f)). We then have, by Lemma \ref{lem: 2.1}, two bounded 
$\BB$-measurable functions $\tau_\R : B \to \R_+^*$ and $\varphi: B\to
\R$ such that 
\eq{eq: 5.4}{
\theta_\R = \tau_\R + \varphi \circ T - \varphi.
}
Denote by $\theta_M(b)$ the $M$-component of $\theta(b)$, and
$\tau_M(b) = \theta_M(b)^{-1}$ and 
\eq{eq: 5.5}{
\tau = (\tau_\R, \tau_M): B \to \R \times M.
}
It  is the suspension $B^\tau$ associated with this function $\tau$
which we will use below. 

This suspension allows us to control the norm of the words which
appear in the formulas for the conditional measures, thanks to the
following lemma. 

\begin{lem}\name{lem: 5.4}
For $\beta$-a.e. $b \in B$, for every $w \in V_b$, we have 
\eq{eq: 5.6}{
\left \| b_0^{-1} w\right \| = e^{-\theta_\R(b)} \|w\|.
}
\end{lem}

\begin{proof}
By the definition of $\theta$, for $\beta$-a.e. $b \in B$, we have
$$
b_0 s(\xi(Tb)) = s(\xi(b)) \theta(b).
$$
Since $w$ is in $V_b$, we can write $w = s(\xi(b))v$ with $v \in
V_0$. We note that this expression makes sense because $U$ acts
trivially on $V_0$. Since the norm is $K$-invariant, we have 
$$
\left \| b_0^{-1} w\right\| = \left \| b_0^{-1} s(\xi(b))v\right \| =
\|\theta(b)^{-1} v\| = e^{-\theta_\R(b)} \|v\|
= e^{-\theta_\R(b)} \|w\|.$$
\end{proof}

\subsection{Behavior of random walks}
{\em We continue our study of the asymptotic behavior of the random
  walk on $G$.}

We will use Proposition \ref{prop: 5.2} to control the drift in Lemma
  \ref{lem: 7.3}, in the form of the following Corollary. 

\begin{cor}\name{cor: 5.5}
\begin{itemize}
\item[a)]
For any $\alpha>0$, there are $r_0 \geq 1$, $q_0 \geq 1$, such that
for any $v \in V \sm \{0\}$, we have 
$$
\beta \left\{a \in B: \forall q \geq q_0, \|a_q \cdots a_0v\| \geq
\frac{1}{r_0} \|a_q \cdots a_0\| \|v\| \right \} \geq 1-\alpha.
$$
\item[b)]
For every $\alpha > 0$ and $\eta>0$, there exists $q_0 \geq 1$, such
that, for every $v \in V\sm \{0\}$, and every $W \in \Gr_{d_0}(V)$, we
have 
$$
\beta\{a \in B: \forall q \geq q_0, \, d (\R a_q \cdots a_0v, a_q
\cdots a_0 W) \leq \eta\} \geq 1-\alpha.
$$
\combarak{apparently $d$ is a distance on the
  projective space of $V$ although this was not defined.}
\end{itemize}
\end{cor}

In order to prove the corollary, we will need the following lemma in
linear algebra. Denote 
$$
O_{d_0}(V) = \{ \pi \in \End(V): \rank (\pi) = d_0 \text{ and }
\pi|_{(\ker \pi)^{\perp}} \text{ is an isometry} \}.
$$
This is a compact subset of $\End(V)$. 
\begin{lem}\name{lem: 5.6}
\begin{itemize}
\item[a)]
For any $\vre>0$, there are $r_0 \geq 1$, $\vre'>0$ such that, for any
$g \in \GL(V)$ and $\pi \in O_{d_0}(V)$ with $\|g-\pi\| < \vre'$, for
any $v \in V \sm \{0\}$ with $d(\R v, \ker \pi) \geq \vre$ we have $\|gv\|
\geq 
\frac{1}{r_0} \|v\|.$ 
\item[b)]
For any $\vre>0$ and $\eta>0$, there is $\vre'>0$ such that, for every
$g \in \GL(V)$ and $\pi \in O_{d_0}(V)$ with $\|g-\pi\| \leq \vre'$ we
have, for all $v \in V \sm \{0\}$ and $W \in \Gr_{d_0}(V)$, if $d(\R v,
\ker \pi) \geq \vre$ and $\inf_{w \in W\sm \{0\}} d(\R w, \ker \pi)
\geq \vre$, then $d (\R gv, gW) \leq \eta.$
\end{itemize}
\end{lem}

\begin{proof}
a). 
Otherwise, we can find sequences $\pi_n$ in $O_{d_0}(V), \ g_n \in
\GL(V)$ and $v_n \in V$ with $\|v_n\|=1$, such that $\|g_n - \pi_n\|
\to 0$, $d(\R v_n, \ker \pi_n) \geq \vre$ and $\|g_n v_n\| \to 0.$ By
compactness, we can assume by passing to subsequences that the
$\pi_n$ converge to $\pi \in O_{d_0}(V)$ and $v_n$ converge to $v \in
V$, $\|v\|=1$. Our assertions imply that $v$ is simultaneously in
$\ker \pi$ and is of distance at least $\vre$ from $\ker \pi$, a
contradiction. 

b). 
The argument is similar to the one used for proving a). 
\end{proof}

\begin{proof}[Proof of Corollary \ref{cor: 5.5}] a) By Proposition
  \ref{prop: 5.2}(d), for any $\alpha > 0$, there is $\vre>0$ such
  that for any $v \in V \sm \{0\}$, 
$$
\beta\{a \in B: d(\R v, V'_a) \geq \vre\} \geq 1-\alpha/2.
$$
\combarak{I think this is correct but in the text the inequality in
  the def of the set on the lhs is reversed.}

On the other hand, by Proposition \ref{prop: 5.2}(b), for any
$\vre'>0$, there is $q_0 \geq 1$ such that 
$$
\beta\left\{a \in B: \forall q \geq q_0, d \left(\frac{a_q \cdots
      a_0}{\|a_q \cdots a_0\|} , O_{d_0}(V) \right) < \vre' \right
\} \geq 1-\alpha/2. 
$$
\combarak{Again I think this is correct but in the text the inequality
  in the definition of the set is reversed.}
It now suffices to apply Lemma \ref{lem: 5.6}(a). 

b) By Proposition \ref{prop: 5.2}(e), for any $\alpha>0$, there is
$\vre>0$ such that for $W \in \Gr_{d_0}(V), $
$$
\beta \left\{a \in B: \inf_{w \in W \sm \{0\}} d(\R w, V'_b) \geq \vre
\right\} \geq 1-\alpha/2.
$$
It suffices to apply, as above, Proposition \ref{prop: 5.2}(b) and
Lemma \ref{lem: 5.6}(b). 
\end{proof}

\section{Homogeneous spaces of semi-simple groups}\name{sec: 6} 
{\em This chapter collects diverse ergodic properties of the random
  walk on homogeneous spaces. These properties will enable us in \S
  \ref{sec: 7} to develop the exponential drift argument. }

\subsection{Notations} \name{subsec: 6.1}
For the proofs of Theorem \ref{thm: 1.1} and \ref{thm: 1.3} we will
use the same method, and common notation.  

WE KEEP THE FOLLOWING NOTATION FOR THE REST OF THE PAPER. 

{\bf In the first case,} i.e. the case of Theorem \ref{thm: 1.1}, $G$
is a connected almost-simple Lie group and $\Lambda$ is a lattice in
$G$. We denote by $X$ the quotient $G/\Lambda$ and by $R$ the Adjoint
representation of $G$ on $V = \mathfrak{g}$, the Lie algebra of $G$. 

{\bf In the second case}, i.e. the case of Theorem \ref{thm: 1.1}, $G$
is the Zariski closure of $\Gamma_\mu$ in $\SL_d(\R)$. We denote by
$X$ the torus $\TT^d$ and by $R$ the representation of $G$ on $V
=\R^d$, that is, the natural action by matrix multiplication,  which
we can think of as the Lie algebra of $\TT^d$.  

{\bf In both cases,} $G$ is a semisimple Lie group (we will give more
details about this in Lemma \ref{lem: 8.5}), the representation $R$ of
$G$ on $V$ is strongly irreducible, $\mu$ is a compactly supported
probability measure such that the subsemigroup $\Gamma = \Gamma_\mu$
generated by $\supp \, \mu$ is Zariski dense in $G$, $\nu$ is a
non-atomic $\mu$-stationary Borel probability measure on $X$ and
$\tau$ is the map given by \equ{eq: 5.5}. We also suppose that $G$ is
not compact (the very easy case in which $G$ is compact is discussed
in Lemma \ref{lem: 8.4}). 

The proof, which we will give from here to the end of the paper,
relies on the properties of the dynamical systems
$$
\left(B^X, \BB^X, \beta^X, T^X \right) \text{ and } \left( B^{\tau,
  X}, \BB^{\tau, X}, \beta^{\tau, X}, T^{\tau, X} \right)
$$
which we introduced in sections \S \ref{subsec: 3.1} and \S
\ref{subsec: 3.3}, for these values of $G,V,X, \tau, \ldots.$ 

\subsection{Recurrence off the diagonal}
{\em We now verify condition (HC) of \S \ref{subsec: 3.4}, which will
  allow us to apply Proposition \ref{prop: 3.9}.}
 
For any $x \in X$, denote by $r_x$ the radius of injectivity at $x$,
that is the least upper bound of $r>0$ such that the map $V \to X$, $w
\mapsto e^w x$ is injective on the ball $B(0,r)$. 

\begin{prop}\name{prop: 6.1}
In the two cases of \S \ref{subsec: 6.1}, the averaging operator
$A_\mu$ on $X \times X$ satisfies condition (HC). 
\end{prop}

The proof of this proposition uses ideas of Eskin and Margulis
\cite{ref 9}. We note the contrast between Proposition
\ref{prop: 6.1} and Theorem 1 of LePage in \cite{ref 26}, who shows
that on the flag variety, a positive power of the
distance is contracted under convolution. We will need the
following two lemmas. We will use the same notation $A_\mu$ to denote all
the averaging operators of $\mu$ on every space on which $\Gamma_\mu$
acts. The first lemma, due to Eskin and Margulis, exhibits a function
on which $A_\mu$ acts by contraction. 

\begin{lem}[\cite{ref 9}] \name{lem: 6.2}
Let $V = \R^d$ and let $G$ be a semi-simple Lie subgroup of $\GL(V)$
such that, for any nonzero $G$-invariant subspace $V' \subset V$, the
image of $G$ in $\GL(V')$ is not compact. Denote by $\varphi$ the
function $\varphi: V \sm \{0\} \to \R^*, \, v \mapsto
\|v\|^{-1}$. Then there is $a_0 < 1, \, \delta_0>0$ and $n_0 \geq 1$,
such that
\eq{eq: 6.1}{
A^n_\mu(\varphi^\delta) \leq a_0^n \varphi^\delta, \text{ for any }
\delta \leq \delta_0 \text{ and } n \geq n_0.
} 
\end{lem}

\begin{proof}
This is Lemma 4.2 of \cite{ref 9}. It is proved by developing the
second order term of $e^{-\delta \log(\|gv\|/\|v\|)}$ and using the
theorem of Furstenberg and Kesten on the positivity of the first
Lyapunov exponent $\lambda_1$. 
\end{proof}

Whenever $X$ is noncompact, we will need a variant of a Lemma of Eskin
and Margulis which shows the existence of a proper function on $X$ which is
contracted, with a fixed constant, by the averaging operator. 

\begin{lem}\name{lem: 6.3}
Let $G$ be a real semisimple connected Lie group without compact
factors, let $\Lambda$ be a lattice in $G$, let $X = G/\Lambda$, and
let $\mu$ be a compactly supported probability measure on $G$ whose
support generates a Zariski-dense semigroup. Then there is a proper
function $u: X \to [0, \infty)$ and constants $a<1, C>0$ and
  $\kappa>0$, such that 
\eq{eq: 6.2}{
A_\mu(u) \leq au+C
}
and, for every $x \in X$, 
\eq{eq: 6.3}{
u(x) \geq r_x^{-\kappa
}. 
}
\end{lem}

\begin{proof}
Since the center of $G$ intersects $\Lambda$ in a finite-index
subgroup, we may assume with no loss of generality that $G$ is adjoint
and hence linear. In \S 3.2 of \cite{ref 9}, a proper function $u$
satisfying \equ{eq: 6.2} is constructed explicitly. Due to this
construction, if we regard $G$ as a group of matrices, there exist
constants $C_0 >0$ and $\kappa_0 > 0$ such that, for every $x =
g\Lambda \in X$, we have the lower bound
\eq{eq: 6.4}{
u(x) \geq C_0 \min_{\gamma \in \Lambda}\|g\gamma\|^{\kappa_0}.
} 
Therefore it suffices to note that there exist constants $C_1>0$ and
$\kappa_1>0$ such that, for every $x = g\Lambda \in X$, we have the
lower bound
\eq{eq: 6.5}{
r_x \geq C_1 \left(\min_{\gamma \in \Lambda} \|g\gamma\|\right)^{-\kappa_1}.
}
In fact, if $h = e^w$ is a nontrivial element of $G$ with $hx=x$, then
for any $\gamma \in \Lambda$, $\delta = \gamma^{-1} g^{-1} hg\gamma$
is in $\Lambda$ and we have 
$$
\|h-e\| \geq \|\delta -e\|\, \|\Ad(g\gamma)^{-1} \|^{-1}.
$$ 
Thus, when denoting $C_2 = \min_{\delta \in \Lambda \sm \{e\}}
\|\delta - e\|$, we have 
$$
\min_{hx=x, h \neq e} \|h-e\| \geq C_2 \left(\min_{\gamma \in \Lambda}
\left \|\Ad(g\gamma)^{-1} \right\| \right)^{-1}. 
$$
The lower bound \equ{eq: 6.5} follows.
\end{proof}

\begin{proof}[Proof of Proposition \ref{prop: 6.1}]
First we remark that if the condition (HC) is satisfied for some power
$\mu^{*n_0}$, then it is satisfied for $\mu$. We choose $a_0<1,
\delta \in (0,1)$ and $n_0 \geq 1$ as in Lemma \ref{lem: 6.2}. By
replacing $\mu$ with $\mu^{*n_0}$, we can assume that $n_0=1$. Let
$\delta \leq \delta_0$. 

For any $x \neq x'$ in $X$, we denote by $r_{x,x'} = \frac{1}{2}
\min (r_x, r_{x'}),$ 
$$
d_0(x,x') = \left\{\begin{matrix} \|w\| & \text{if } x'=e^wx \text{
  with } w \in V, \|w\| \leq r_{x,x'} \\ 
r_{x,x'} & \text{otherwise} \end{matrix} \right. , 
$$ 
$$
v_0(x,x') = d_0(x,x')^{-\delta}.
$$
Whenever $X$ is compact, the function $v=v_0$ can be used. In the
general case, we introduce the function $u$ and constant $a<1,
C>0$ and $\kappa>0$ given by Lemma \ref{lem: 6.3}. We may suppose that
$a=a_0$. We set $R_0 = \sup_{g \in \supp \, \mu} \max \left(\| R(g)\|,
\|R(g)^{-1}\| \right).$ If one chooses $\delta < \kappa$ and $C_0 =
\frac{2R_0^{2\delta}}{1-a_0}$, then the function $v$, given for any $x
\neq x'$ in $X$ by 
\eq{eq: 6.6}{
v(x,x') = v_0(x,x') + C_0(u(x)+u(x')),
}
satisfies condition (HC).

In fact, if $d_0(x,x') \geq R^{-1}_0 r_{x,x'}$ then by \equ{eq: 6.3}, 
\[
\begin{split}
(A_\mu v_0)(x,x') & \leq R^{2\delta}_0 r^{-\delta}_{x,x'} \\ 
& \leq 2R_0^{2\delta} (r_x^{-\delta} + r_{x'}^{-\delta}) \leq
2R_0^{2\delta}(u(x)+u(x')). 
\end{split}
\]
On the other hand, if $d_0(x,x') \leq R_0^{-1} r_{x,x'}$, then, when
writing $x' - e^wx$ with $w \in V$, $\|w\| \leq r_{x,x'}$, we have,
for any $g \in G$ of norm at most $R_0$, 
$$
v_0(gx, gx') = \|gw\|^{-\delta},
$$
and hence, by \equ{eq: 6.1}, 
$$
(A_\mu v_0)(x,x') \leq a_0 \|w\|^{-\delta} = a_0 v_0(x,x').
$$
In both cases, we have therefore the upper bound 
$$
(A_\mu v_0)(x,x') \leq a_0v_0(x,x')+R_0^{2\delta} (u(x)+u(x')).
$$
Inequality \equ{eq: 6.2} and the definition \equ{eq: 6.6} of $v$ thus
give the upper bound 
\[
\begin{split}
(A_\mu v)(x,x') & \leq a_0 v_0(x,x')+(R_0^{2\delta}+
  a_0C_0)(u(x)+u(x'))+2CC_0 \\
& \leq \frac{1+a_0}{2} v(x,x') + 2CC_0,
\end{split}
\]
which yields property (HC).
\end{proof}

\subsection{Recurrence off of finite orbits}
{\em In this section we exhibit the phenomenon of recurrence away from
  finite orbits for random walks on $X$, analogous to the phenomenon
  of recurrence to compact subsets in \cite{ref 9}.} 

\begin{prop}\name{prop: 6.4}
In the two cases of \S \ref{subsec: 6.1}, let $F$ be a finite
$\Gamma$-invariant set. Then for any $\vre>0$, there is a compact
subset $K_\vre$ of $F^c$ such that for any $x \in X \sm F$, there is a
constant $M = M_x$, which can be chosen to be uniform for $x$ in a
compact subset of $X \sm F$, such that for all $n \geq M$, 
$$
A^n_\mu(1_{K_\vre}) \geq 1-\vre.
$$ 
\end{prop}

We will need the following two lemmas. 

The first translates the phenomenon of recurrence to compact subsets,
due to Foster, and utilized in this context by Eskin and Margulis. 

\begin{lem}[\cite{ref 9}] \name{lem: 6.5} 
Let $H$ be a locally compact group acting continuously on a locally
compact space $Y$, and let $\mu$ be a Borel probability measure on
$H$. 

Suppose that there is a proper map $f: Y \to [0, \infty)$, and
constants $a<1, b>0$ such that $A_\mu(f) \leq af + b.$ 

Then for any $\vre>0$ there is a compact $K \subset Y$ such that for
every $y \in Y$, there is a constant $M_y$, which can be chosen to be
uniform in $y$ for $y$ in a compact subset of $Y$, such that for all
$n \geq M$, 
$$
A^n_\mu(1_K) \geq 1-\vre.
$$
\end{lem}
We recall the short proof of this lemma. 

\begin{proof}
By hypothesis, we have for each $n \geq 1$, 
$$
A^n_\mu(f) \leq a^nf + b(1+ \cdots a^{n-1}) \leq a^nf+B,
$$
where $B = \frac{b}{1-a}.$ Since $f$ is proper, we can choose as our compact
subset 
$$
K = \left \{z \in Y: f(z) \leq \frac{2B}{\vre} \right \}
$$
which impies that $1^{K^c} \leq \frac{\vre}{2B} f.$ Therefore we have
the upper bounds 
$$
A^n_\mu(1_{K^c}) \leq \frac{\vre}{2B} A^n_\mu(f)(y) \leq \frac{\vre
  a^n}{2B} f(y) + \frac{\vre}{2} \leq \vre,
$$
whenever $n$ is sufficiently large so that $f(y) \leq \frac{B}{a^n}.$ 
\end{proof}

The second Lemma is a variant of Proposition \ref{prop: 6.1}. 

\begin{lem}\name{lem: 6.6}
In the two cases of \S \ref{subsec: 6.1}, let $F \subset X$ be a
finite $\Gamma$-invariant subset. Then there is a proper map $u_F : X
\sm F \to [0, \infty)$ and constants $a<1, C>0$ such that 
\eq{eq: 6.7}{
A_\mu(u_F) \leq au_F + C.
}
\end{lem}

\begin{proof}
We proceed as in the proof of Proposition \ref{prop: 6.1}. We choose
$a_0<1$, $\delta_0>0$ and $n_0 \geq 1$ as in Lemma \ref{lem: 6.2}. By
replacing $\mu$ with $\mu^{* n_0} $ if necessary, we may assume that
$n_0 =1$. Let $\delta \leq \delta_0$. 

Let $r_0>0$ be a real number such that for every $x_0 \in F$, there is
$r_0 \leq \frac{1}{2} r_{x_0} $ such that for every pair $x_0, x'_0$
of distinct points of $F$, $ r_0 \leq \frac{1}{2} d_0(x_0, x'_0)$. For
any $x \in X$, we denote
$$
d_0(x) = \left \{\begin{matrix} \|w\| & \text{if } x=e^wx_0 \text{
      with } x_0 \in F \text{ and } \|w\| \leq r_0 \\ 
r_0 & \text{otherwise} \end{matrix} \right. 
$$
and 
$$
u_0(x) = d_0(x)^{-\delta}.
$$
Whenever $X$ is compact, the function $u_F = u_0$ satisfies the
requirements. In the general case, the function $u_F = u_0 + u$ as in
Lemma \ref{lem: 6.3} satisfies the requirement. The presence of $u$ is
needed only to assure the property of $u_F$. To check that $u_F$
satisfies the requirements, we set 
$$R_0 = \sup_{g \in \supp \, \mu}
\max \left( \|R(g)\|, \|R(g)^{-1}\| \right).$$

On one hand, if $d_0(x) \geq R_0^{-1} r_0$ then we have 
$$
(A_\mu u_0) \leq R^{2\delta}_0 r_0^{-\delta}.
$$
On the other hand, if $d_0(x) \leq R_0^{-1}r_0$ then, when writing $x
= e^wx_0$ with $x_0 \in F$, we have for every $ g \in G$ of norm
at least $R_0$, 
$$
d_0(gx) \leq \|gq\|^{-\delta},
$$
and thus, by \equ{eq: 6.1}, 
$$
(A_\mu u_0)(x) \leq \|w\|^{-\delta} = a_0u_0(x). 
$$
In all cases, we have the upper bound 
$$
(A_\mu u_0)(x) \leq a_0 u_0 (x) + R^{2\delta}_0r_0^{-\delta}. 
$$
This inequality and that of Lemma \ref{lem: 6.3} provide the
sought-for inequality concerning $u_F$. 
\end{proof}

\begin{proof}[Proof of Proposition \ref{prop: 6.4}]
This follows from Lemma \ref{lem: 6.5} applied to $Y = X \sm F$ and
to the function $f = u_F$ of Lemma \ref{lem: 6.6}. 
\end{proof}

\subsection{Stationary probability measures on $G/H$.} 
{\em In order to exploit the drift argument, we will need, in the
  first case of \S \ref{subsec: 6.1}, the following proposition which
  is of independent interest. }

\begin{prop}\name{prop: 6.7}
Let $G$ be a connected semi-simple real Lie group without compact factors, $\mu$
a compactly supported probability measure whose support generates a
Zariski dense subsemigroup in $G$, and $H \subset G$ a unimodular
subgroup. If there exists a $\mu$-stationary probability measure on
the homogeneous space $G/H$, then the Lie algebra of $H$ is an ideal
in the Lie algebra of $G$. 
\end{prop}

For the proof, we will use the following lemma. 
\combarak{It would be nice to prove this with assuming semisimplicity
  of $G$. There is an argument in David's notes.}

\begin{lem}\name{lem: 6.8}
Let $V = \R^d$, let $G$ be a semi-simple subgroup of $\GL(V)$ with no
compact factors, and let $\mu$ be a compactly supported Borel
probability measure on $G$ generating a Zariski dense
subsemigroup. Then any $\mu$-stationary probability measure $\nu$ on
$V$ is supported on the subspace $V^G$ of $G$-fixed points in $V$. 
\end{lem}

\begin{proof}
Suppose by contradiction that there is a $\mu$-stationary probability
measure $\nu$ on $V$ which is not supported on $V^G$. Then there is an
irreducible sub-representation $W \subset V$ of dimension at least 2
such that the projection of $\nu$ on $W$ is not a Dirac mass at
0. This projection is also $\mu$-stationary. Thus we may assume that
$V$ is irreducible and $G$ is not compact. 

We will use again the Bernoulli system $(B, \BB, \beta, T)$ with
alphabet $(G, \mu)$ and the fibered dynamical system $B \times V$
equipped with the transformation $\mathbf{R}: (b,v) \mapsto (Tb, b_0v)$ which
leaves the probability measure $\beta \otimes \nu$
invariant. \combarak{In the text what is here called $\mathbf{R}$, is
  called $R$. I distinguish the two since in \S \ref{subsec: 6.1}, the
  symbol $R$ is used for a representation.}

The theorem of Furstenberg and Kesten about the positivity of the
first Lyapunov exponent (\cite{ref 15}, see also \cite{ref 12},
chapter 1) ensures that for $\beta$-a.e. $b \in B$, there is a
subspace $W_b \varsubsetneq V$ such that, for any $v \in V \sm W_b$,
the norm $\|b_n \cdots b_0 v\|$ converges (exponentially fast) to
infinity. We introduce the $\mathbf{R}$-invariant set 
$$
Z = \{(b,v)  \in B \times V : v \notin W_b\}
$$
and the function $\varphi$ on $Z$ given by 
$$
\varphi(b,v) = \|v\|.
$$
Since $\nu$ is $\mu$-stationary, and since $\mu$ is Zariski dense in
$G$ and the action of $G$ on $V$ is irreducible, $\nu$ does not give
positive mass to any proper subspaces of $V$. We therefore have
$(\beta \otimes \nu)(Z)=1$. By construction, for $\beta \otimes
\nu$-a.e. $z \in Z$, we have 
$$
\lim_{n \to \infty} \varphi(\mathbf{R}^nz) = \infty, 
$$
which contradicts the Poincar\'e recurrence theorem. 
\end{proof}

\begin{proof}[Proof of Proposition \ref{prop: 6.7}]
We denote by $\nu$ a $\mu$-stationary probability measure on $G/H$,
denote by $\mathfrak{g}$ the Lie algebra of $G$, by $\mathfrak{h}$
that of $H$, set $r = \dim \mathfrak{h}$, $V = S^2(\bigwedge^r
\mathfrak{g})$ and let $v$ be a nonzero point of the line
$S^2(\bigwedge^r \mathfrak{h}) \subset V$. \combarak{In the text it
  says $S^2(\bigwedge^d \mathfrak{h}) \subset V$ but I believe this
  wrong, $d$ should be $r$.}

Since $H$ is unimodular, $H$ is contained in the stabilizer $N$ of the
point $v$. Therefore the orbit $Gv \cong G/N$ also admits a stationary
measure: the image $\nu'$ of $\nu$ under the projection $G/H \to
G/N$. By Lemma \ref{lem: 6.8}, $\nu'$ is supported on the subspace
$V^G$ of $G$-fixed vectors. Thus $N=G$. Since $N$ normalizes
$\mathfrak{h}$, $\mathfrak{h}$ is an ideal of $\mathfrak{g}$. 
\end{proof}

\subsection{Horocycle flows}
{\em The goal of this section is to construct an action of $V_0$ which
  plays a role analogous to the one played by the horocycle flow on
  compact hyperbolic surfaces, in the sense that the orbits of this
  action are contained in the stable leaves relative to the factor map
  $B^{\tau, X} \to B^\tau$ and they are uniformly dilated by the
  semi-flow $T^\tau$. }

We keep the notations of \S \ref{subsec: 6.1}. 

\begin{dfn}
The horocycle flow is the action $\Phi$ of $V_0$ on $B^{\tau, X}$
given by, for any $v \in V_0$ and $\beta^\tau$-a.e. $c =
(b,k,m) \in B^\tau$ and every $x \in X$, 
\eq{eq: 6.8}{
\Phi_v(c,x) = (c, \exp(D_c(v))x),
}
where $D_c(v)$ is the element of $V_c$ given by 
\eq{eq: 6.9}{
D_c(v) = e^{k - \varphi(b)} s(\xi(b)) mv.
}
\end{dfn}

Recall that $s, \xi, \varphi$ were defined in \S \ref{subsec: 5.1} and
\S \ref{subsec: 5.2}. Geometrically, the flow $\Phi$ `translates every
point $(c,x)$ in the direction of $V_c$'. We note that at this stage
in the argument, we do not know that this flow preserves the
probability measure $\beta^{\tau,X}$: we will know this after having
proved Theorem \ref{thm: 1.1}. This difficulty is certainly a source
of complications which are the heart of the matter. 

The fundamental property of the horocycle flow is its relationship
with the flow $(T_\ell^{\tau, X})_{\ell \geq 0}$ on $B^{\tau, X}$. 

\begin{lem}\name{lem: 6.10}
In the two cases of \S \ref{subsec: 6.1}, for any $v \in V_0$ and any
$\ell \geq 0$, we have, for $\beta^\tau$-a.e. $c \in B^\tau$ and any
$x \in X$, 
\eq{eq: 6.10}{
T_\ell^{\tau, X} \circ \Phi_v (c,x) = \Phi_{e^{-\ell}v} \circ
T_\ell^{\tau, X}(c,x).
}
\end{lem}

\begin{proof}
Denote by $S$ the transformation of $B \times \R \times M \times X$
given by 
$$
S(b,k,m,x) = (Tb, k-\tau_\R(b), \tau_M(b)m, b_0^{-1}x).
$$
We note that $B^{\tau, X}$ is the set of points in $B \times \R_+
\times M \times X$ which are taken by $S$ to points outside of this product. 
\combarak{i.e. points $(b,k,m,x)$  with $k \geq 0$ such that after applying $S$,
  the second term becomes negative.} 

Introduce the flow $\til T_\ell^{\tau, X}$ defined on $B \times \R
\times M \times X$ by 
$$
\til T_\ell^{\tau, X}(b,k,m,x) = (b, k+\ell, m, x).
$$
The flow $T_\ell^{\tau, X}$ is given, for $\ell \geq 0$ and $(b,k,m,x)
\in B^{\tau, X}$, by 
$$
T_\ell^{\tau, X} (b,k,m,x) = (S^p \circ \til T_\ell^{\tau, X})(b,k,m,x)
$$
where $p \geq 0$ is the unique integer for which this expression is in
$B^{\tau, X}$. We then define an action $\til \Phi$ of $V_0$ on $B
\times \R \times M \times X$ by the formula:
$$
\til \Phi_v(b,k,m,x) = (b,k,m, \exp(D_{(b,k,m)}(v))x)
$$
where 
\eq{eq: 6.11}{
D_{(b,k,m)}(v) = e^{k-\varphi(b)} s(\xi(b)) mv.
}
Before continuing we prove the following equality: for
$\beta$-a.e. $b \in B$, every $(k,m) \in \R \times M$, and every $v
\in V_0$, we have
\eq{eq: 6.12}{
b_0^{-1} D_{(b,k,m)} (v) = D_{S(b,k,m)}(v)
}
where 
\eq{eq: 6.13}{
S(b,k,m) = (Tb, k -\tau_\R(b), \tau_M(b)m).
}
To this end, we compute as in Lemma \ref{lem: 5.4}, 
\[
\begin{split}
b_0^{-1} D_{(b,k,m)}(v) & = e^{k-\varphi(b)}b_0^{-1} s(\xi(b)) m v \\ 
& = e^{k-\varphi(b)} s(\xi(Tb)) \theta(b)^{-1} mv \\ 
& = e^{k -\varphi(b)-\theta_\R(b)} s(\xi(Tb))\theta_M(b) mv , 
\end{split}
\]
and hence, using \equ{eq: 5.4}, 
\[
\begin{split}
b_0^{-1} D_{(b,k,m)}(v) & = e^{k-\tau_\R(b) - \varphi(Tb)}
s(\xi(Tb))\tau_M(b) mv \\ 
& = D_{S(b,k,m)}(v).
\end{split}
\]
We deduce, thanks to \equ{eq: 6.12}, the following two equalities
\eq{eq: 6.14}{
S \circ \til \Phi_v = \til \Phi_v \circ S
}
and 
\eq{eq: 6.15}{
\til T_\ell^{\tau, X} \circ \til \Phi_v = \til \Phi_{e^{-\ell}v} \circ
\til T_\ell^{\tau, X}
}
which proves that the flow $\Phi_v$ satisfies \equ{eq: 6.10}. 
\end{proof}

\subsection{Horocyclic conditional probabilities}
{\em In this section we introduce the `horocyclic conditional
  function' and prove that this function is measurable for the tail
  $\sigma$-algebra.}

We keep the notations of \S \ref{subsec: 6.1} and also denote by $t_v$
the translation of $V_0$ by an element $v \in V_0$. We write $\sigma:
B^{\tau, X} \to \MM_1(V_0)$ the map given by `conditional measures of
the probability measure $\beta^{\tau, X}$ with respect to the
horocyclic action of $V_0$'. 

\begin{lem}\name{lem: 6.11}
In the two cases of \S \ref{subsec: 6.1}, there is a Borel subset $E
\subset B^{\tau, X}$ such that $\beta^{\tau, X}(E^c)=0$ and such that,
for any $v \in V_0$ and $(c,x) \in E$ for which $\Phi_v(c,x) \in E$,
we have 
\eq{eq: 6.16}{
t_{v*}\sigma(\Phi_v(c,x)) \simeq \sigma(c,x). 
}
\end{lem}
\begin{proof}
This follows from Proposition \ref{prop: 4.2}. 
\end{proof}

Recall that the symbol $\simeq$ refers to equality after a
normalization by a scalar. 

Geometrically, for $\beta^{\tau, X}$-a.e. $(c,x) \in B^{\tau, X}$,
$\sigma(c,x)$ is the conditional measure of $\delta_c \otimes \nu_c$
for the action of $V_0$ on $\{c\} \times X$. 

\begin{lem}\name{lem: 6.12}
In the two cases of \S \ref{subsec: 6.1}, for any $\ell \geq 0$, for
$\beta^{\tau, X}$-a.e. $(c,x) \in B^{\tau, X}$, we have 
$$
\sigma(T_\ell^{\tau, X}(c,x)) \simeq (e^{-\ell})_* \sigma(c,x). 
$$
\end{lem}
In this equality, $e^{-\ell}$ denotes the homothety by a factor of
$e^{-\ell}$ of $V_0$. 

\begin{proof}
This is a result of the uniqueness  of $\sigma$, equality \equ{eq:
  6.10} and the fact that for $\beta$-a.e. $b \in B$, for any $p \in
\N$, the action of $b_{p-1}^{-1} \cdots b_0^{-1}$ induces an
isomorphism between the measure spaces $(X, \nu_b)$ and $(X,
\nu_{T^pb})$. 
\end{proof}

\begin{cor}\name{cor: 6.13}
In the two cases of \S \ref{subsec: 6.1}, the map $\sigma: B^{\tau,X}
\to \MM_1(V_0)$ is $\QQ^{\tau, X}_\infty$-measurable. 
\end{cor}
\begin{proof}
It suffices to show that for any $\ell \geq 0$, it is $\QQ^{\tau,
  X}_\ell$-measurable. This results from the equality, for
$\beta^{\tau, X}$-a.e. $(c,x) \in B^{\tau, X}$, $\sigma(c,x) \simeq
(e^\ell)_*(\sigma(T^{\tau, X}_\ell(c,x))).$ 
\end{proof}

\subsection{Approach outside the $W$-leaves}
{\em In order to start the drift argument, we need to ensure, in any
compact subset of positive $\beta^{\tau, X}$-measure, that a.e. point
$x$ is approached by points which are not in the same leaf as $x$ for
a certain  subfoliation of the relative stable leaf. }

For $b \in B$, we introduce a vector subspace of $V$:
\eq{eq: 6.17}{
W_b = \left\{v \in V: \sup_{p \in \N} \left( e^{\theta_{\R, p}(b)}
    \|b_p^{-1} \cdots b_0^{-1} v\| \right) < \infty \right \}
}
and, for $c = (b,k,m) \in B^\tau$, we set $W_c  = W_b$. 

\begin{lem}\name{lem: 6.14}
In the two cases of \S \ref{subsec: 6.1}, for $\beta^X$-a.e. $(b,x)$
in $B^X$, we have $\nu_b(\exp(W_b)x)=0$. 
\end{lem}
\begin{proof}
By ergodicity of the Bernoulli system $(B, \BB, \beta, T)$ and by
Furstenberg's formula \equ{eq: 5.3}, for $\beta$-a.e. $b \in B$ we
have $\lim_{p \to \infty} \frac{1}{p} \theta_{\R, p}(b) = \int_B
\theta_{\R}(b) d\beta(b) = \lambda_1>0$. Therefore, by Lemma \ref{lem:
  5.4}, for every $v \in W_b$, we have $\lim_{p \to \infty} \|b_p^{-1}
\cdots b_0^{-1} v\|=0$. Choosing a distance function $d$ on $X$, gives
a right-invariant distance on the group $\til X$, the universal cover
of $X$. For $\beta$-a.e. $b \in B$, every $x \in X$, and every $v
\in W_b$, we have 
$$
d(b_p^{-1} \cdots b_0^{-1}\exp(v)x, b_p^{-1} \cdots b_0^{-1} x) \to_{p
  \to \infty} 0.
$$
By Proposition \ref{prop: 6.1}, the measure $\mu$ satisfies property
(HC), and hence, by Proposition \ref{prop: 3.9}, for $\beta^X$-a.e. $(b,x)
\in B^X$, we have $\nu_b(\exp(W_b)x)=0$, as required. 
\end{proof}

\begin{cor}\name{cor: 6.15}
In the two cases of \S \ref{subsec: 6.1}, let $F \subset
B^{\tau, X}$ be a $\BB^{\tau, X}$-measurable subset such that
$\beta^{\tau, X}(F)>0$. Then, for $\beta^{\tau, X}$-a.e. $(c,x) \in F$,
there is a sequence $(u_n)$ of elements of $V \sm W_c$ such that $u_n
\to 0$ and such that, for every $n$, $(c, \exp(u_n)x) \in F$. 
\end{cor}

\begin{proof}
Let $(U_n)$ be a countable basis of neighborhoods of $0$ in $V$. For
$\beta^\tau$-a.e. $c \in B^\tau$, the set $F_c = \{x \in X: (c,x) \in
F\}$ satisfies $\nu_c(F_c)>0$. For $\beta^{\tau, X}$-a.e. $(c,x) \in
F$, for every $n \geq 0$ we therefore have $\nu_c(F_c \cap
\exp(U_n)x)>0$ and since, by Lemma \ref{lem: 6.14},
$\nu_c(\exp(W_c)x)=0$, we have $\nu_c(F_c \cap (\exp(U_n \sm
W_c)x))>0$. 
\end{proof}

\section{Invariance of stationary measures}\name{sec: 7}
{\em The goal of this chapter is to present the exponential drift
  argument and to deduce invariance properties for certain conditional
  measures of stationary measures (Proposition \ref{prop: 7.6}).

To this end we collect the pieces of the puzzle which we have
prepared in previous chapters.}

\subsection{The exponential drift}\name{subsec: 7.1}
{\em The heart of this paper is the following proposition.}

We keep as always the notation of \S \ref{subsec: 6.1}. In particular,
$\mu$ is a probability measure on $G$ whose support generates a
Zariski-dense subsemigroup, $\nu$ is a $\mu$-stationary and
$\mu$-ergodic Borel probability measure on $X$, and the symbols $s,
\xi, \theta, \theta_\R, \varphi, \tau_\R, \tau_M, \tau, B^\tau,
\beta^\tau, \beta^{\tau, X}, \sigma, R$, etc., have the same meanings
as in \S \ref{sec: 5} and \S \ref{sec: 6}. 

\begin{prop}\name{prop: 7.1}
In the two cases of \S \ref{subsec: 6.1}, let $(Y, \YY)$ be a standard
Borel space, let $f: B^{\tau, X} \to Y$ be a $\QQ_\infty^{\tau,
  X}$-measurable map, and let $E \subset B^{\tau, X}$ be a $\BB^{\tau,
  X}$-measurable subset such that $\beta^{\tau, X}(E^c)=0$. Then for
$\beta^{\tau, X}$-a.e. $(c,x) \in B^{\tau, X}$, for any $\vre>0$,
there exists a nonzero element $v \in V_0$ of norm at most $\vre$ and
an element $(c',x') \in E$ such that $\Phi_v(c',x') $ is also in $E$
and such that 
\eq{eq: 7.1}{
f(\Phi_v(c',x')) = f(c',x') = f(c,x). 
}
\end{prop} 

\begin{remark}
Since we do not yet know that the horocycle flow preserves the measure
$\beta^{\tau, X}$ (we will show this in \S \ref{subsec: 8.1}), it is
not apriori clear that {\em there exists an element $(c',x') \in E$ and a
nonzero vector $v \in V_0$ such that $\Phi_v(c',x')$ is in $E$. } This
assertion will be a nontrivial consequence of Proposition \ref{prop: 7.1}.
\end{remark}

{\em Beginning of proof of Proposition \ref{prop: 7.1}.}
By definition, we can assume that $Y$ is endowed with the topology of
a complete separable metric space for which $\YY$ is the Borel
$\sigma$-algebra. Similarly we can choose the topology of a compact
metric space on $B^\tau$ so that the Borel $\sigma$-algebra coincides,
up to adding subsets of measure zero, with $\BB^\tau$, and such that
the natural projection $B^\tau \to M$ is continuous, and endow $B^\tau
\times X$ with the product topology of this topology and the usual
topology on $X$. 

Let $\alpha>0$ be a small number. By Lusin's theorem, there is  a
compact subset $K \subset E$ in $B^{\tau, X}$ such that $\beta^{\tau,
  X}(K^c) < \alpha^2$ and such that all the functions we will
encounter, such as the functions $f$, $\theta$, $(c,x) \mapsto
\varphi(b), \, (c,x) \mapsto V_c$ and also $(c,x) \mapsto D_c \in
\Hom(V_0, V_c)$, are uniformly continuous on $K$. 

The proof relies on the study of the function $\EE \left(1_K |
\QQ_\infty^{\tau, X} \right).$ 

On one hand, this function is bounded above by 1 and its average is
bounded below by $1-\alpha^2$, because:
\eq{eq: 7.2}{
\int_{B^{\tau, X}} \EE\left(1_K | \QQ_\infty^{\tau, X}\right)(c,x) d\beta^{\tau,
  X}(c,x) = \beta^{\tau, X}(K) > 1-\alpha^2.
}
Thus the function $\EE \left(1_K | \QQ_\infty^{\tau, X}\right)$ is bounded below
by $1-\alpha$ on a set of measure  $1-\alpha$. Therefore there is a
compact subset $L \subset E$ in $B^{\tau, X}$ such that $\beta^{\tau,
  X}(L^c) < \alpha$ and such that, for every $(c,x) \in L$, we have
\eq{eq: 7.3}{
\EE \left(1_K | \QQ_\infty^{\tau, X} \right)(c,x) > 1-\alpha.
}
By Lusin's theorem, we may also suppose that $f$ is continuous on
$L$. 

On the other hand, by the Martingale convergence theorem, for
$\beta^{\tau, X}$-a.e. $(c,x) \in B^{\tau, X}$, we have 
\eq{eq: 7.4}{
\lim_{\ell \to \infty} \EE\left(1_K | \QQ_\ell^{\tau, X} \right)(c,x)
= \EE \left(1_K
| \QQ_\infty^{\tau, X} \right)(c,x). 
}
By Corollary \ref{cor: 3.8}, we may also suppose that for
\underline{every} $(c,x) \in L$ and $\ell$ rational, the left hand
side of \equ{eq: 7.4} is given by formula \equ{eq: 3.11}. Thanks to
the law of the last jump (Proposition \ref{prop: 2.3}), recalling the
notation $h_{\ell, c}(a)$, this can be rewritten as 
\eq{eq: 7.5}{
\EE \left(1_K | \QQ_\ell^{\tau, X} \right)(c,x) = \int_B 1_K (h_{\ell, c, x}(a)) d\beta(a),
}
where 
$$
h_{\ell, c,x}(a) = (c',x') \text{ with } c' = h_{\ell, c}(a) \text{
  and } x' = \rho_\ell(c')^{-1} \rho_\ell(c)x.
$$
Moreover, since $f$ is $\QQ_\infty^{\tau, X}$-measurable, it is
$\QQ_\ell^{\tau ,X}$-measurable for each $\ell \geq 0$, and hence,
again by Corollary \ref{cor: 3.8} and 
Proposition \ref{prop: 2.3}, we can also assume that for every $(c,x)
\in K$, for $\beta$-a.e. $a \in B$, for any rational $\ell \geq 0$, we
have $f(h_{\ell, c,x}(a)) = f(c,x).$ 

Egorov's theorem ensures that, outside a subset of $L$ of arbitrarily
small $\beta^{\tau, X}$-measure, the convergence in \equ{eq: 7.4} is
uniform on $L$. Therefore, after removing a subset of $L$ of small
measure, there exists $\ell_0 \geq 0$ such that for 
every integer $\ell \geq \ell_0$, for every $(c,x) \in L$, we have 
\eq{eq: 7.6}{
\EE\left(1_K | \QQ_\ell^{\tau, X} \right)(c,x) \geq 1-\alpha.
}
Since the $\beta^{\tau, X}$-measure of $L^c$ is at most $\alpha$ and
$\alpha$ was chosen arbitrarily small, it suffices to prove \equ{eq: 7.1}
 for $\beta^{\tau, X}$-a.e. $(c,x) \in L$. 

By Corollary \ref{cor: 6.15} we may suppose that for the points $(c,x)
\in L$, there exists a sequence $(u_n)$ of elements of $V \sm W_c$
which converge to 0 and such that the points $(c,y_n)$ defined by
$(c,y_n) = (c, \exp(u_n)x)$ are also in $L$. 

We apply the two formulas \equ{eq: 7.5} and \equ{eq: 7.6} to the
conditional expectations at the two points $(c,x)$ and $(c,y_n)$. For
$\ell \geq \ell_0$, we then have 
\eq{eq: 7.7}{
\beta \left \{ a \in B: h_{\ell, c, x}(a) \in K\right \} \geq 1-\alpha 
}
and 
\eq{eq: 7.8}{
\beta \left \{ a \in B: h_{\ell, c, y_n}(a) \in K \right \} \geq 1-\alpha. 
}
We will now say a few words about the strategy of proof. By
construction, for $y = \exp(u)x$ with $u \in \mathfrak{g}$, the
parameterizations of the two fibers of $T_\ell^{\tau, X}$ passing
through $(c,x)$ and $(c,y)$ are related by a {\em drift} that can be
easily computed: if $(c',x') = h_{\ell, c,x}(a)$ and $(c',y') =
h_{\ell, c,y}(a)$, then we have 
\eq{eq: 7.9}{
y' = \exp(F_{\ell,c}(a)u)x'
}
where the drift is given by 
\eq{eq: 7.10}{
F_{\ell, c}(a)u = R_\ell(c') \circ R_\ell(c)^{-1}(u),
}
and where, as in \S \ref{subsec: 3.3}, if we write $c = (b,k,m)$ and
$p = p_\ell(b,k)$ then we have 
\eq{eq: 7.11}{
R_\ell(c) = R(b_0) \circ \cdots \circ R(b_{p-1}).
}
To simplify the notations, we will sometimes write $b_0$ for
$R(b_0)$. We will see that, for the parameterization of the two fibers
of $T^{\tau, X}_\ell$ passing through the points $(c,x)$ and $(c,
y_n)$, a large proportion of the parameters $a \in B$ correspond to
two points $(c'_n, x'_n)$ and $(c'_n, y'_n)$ which are both in $K$. We
will now adjust the line $\ell = \ell_n$ of the sequence $u_n$ in
order to control the norm and the direction of the drift separating
these two points. 

This will be possible thanks to the following lemma. 

\begin{lem}\name{lem: 7.3}
In the two cases of \S \ref{subsec: 6.1}, for any $\alpha>0$ and
$\eta>0$, there is $r_0 \geq 1$, such that for $\beta^\tau$-a.e. $c
\in B^\tau$, for all $\ell$ sufficiently large, we have for all $u \in
V \sm \{0\}$, 
\eq{eq: 7.12}{
\beta \left \{ a \in B: \frac{1}{r_0} \leq \frac{\|F_{\ell, c}(a) u
    \|}{e^{\theta_{\R, \ell}(c)} \|R_\ell(c)^{-1} u\|} \leq r_0 \right\} \geq 1-\alpha.
}
and 
\eq{eq: 7.13}{
\beta \left \{ a \in B: d\left(\R F_{\ell, c}(a) u, \PPP \left(V_{h_{\ell,
      c}(a)} \right) \right) \leq \eta \right \} \geq 1-\alpha. 
}
\end{lem}

\begin{proof}
Recall that by \S\ref{subsec: 2.3}, for $\beta^\tau$-a.e. $c \in
B^\tau$, for $\beta$-a.e. $a \in B$, we have $\lim_{\ell \to \infty}
q_{\ell,c}(a) = \infty.$ 

In order to obtain the upper bound \combarak{(sic !)} \equ{eq: 7.12}, we apply Corollary
\ref{cor: 5.5}(a) with the vectors  $v_1 = R_{\ell}(c)^{-1}u$ and $v_2
\in V_{T^\tau_\ell(c)}$ which results in the equality 
$$
\frac{\|F_{\ell, c}(a)u\|}{\|R_\ell(c)^{-1}u \|} = \frac{\| a_{q-1}
  \cdots a_0 v_1\|}{\|v_1 \|}
$$
for $q = q_{\ell, c}(a)$, and, thanks to Lemma \ref{lem: 5.4}, in the
equality 
$$
e^{\theta_{\R, \ell}(c')} = \frac{\| a_{q-1} \cdots a_0 v_2\|}{\|v_2 \|}
$$
with $c'=h_{\ell, c}(a)$. 
\combarak{Still need to compare $\theta_{\R, \ell}(c)$ to $\theta_{\R,
    \ell}(c')$. To show they are equal, up to additive factors,
  first use that $\tau_{p_\ell(c)}, \tau_{p_\ell( c')}$ are each equal to
  $\ell$ up to
  additive factors because $\tau$ is bounded, and then use that 
in Lemma \ref{lem: 2.1}, $\varphi$ can be taken bounded.
}
In order to obtain \equ{eq: 7.13}, we apply Corollary \ref{cor:
  5.5}(b) with the same vector $v = R_\ell(c)^{-1}u$ and with $W =
V_{T^\tau_\ell(c)}$. For $\beta^\tau$-a.e. $c \in B^\tau$, for every
$\alpha, \eta>0$ there is thus $\ell_0 \geq 0$ such that for every $u
\in V \sm \{0\}$ and $\ell \geq \ell_0$, 
$$
\beta \left \{ a \in B: d(\R F_{\ell, c}(a) u, \mathbb{P}(V_{h_{\ell,
      c}(a)})) \leq \eta \right \} \geq 1-\alpha, 
$$
\combarak{Again, order seems to be reversed.}
as required. 
\end{proof}

\begin{proof}[End of proof of Proposition \ref{prop: 7.1}]
We now explain our strategy in more detail. We will choose the
parameter $\ell = \ell_n$ in the following manner. 

Since the measure $\mu$ on $G$ is compactly supported, and since the
section $s$ in \S \ref{subsec: 5.2} has a bounded image, there is
$C_0>0$ such that, for $\beta$-a.e. $b \in B$, for every $u \in V \sm
\{0\}$, and every $p \in \N$, we have 
$$
\frac{e^{\theta_{\R, p+1}(b)}\|b^{-1}_{p+1} \cdots b^{-1}_0 u
  \|}{e^{\theta_{\R, p}(b)}\|b^{-1}_p \cdots b^{-1}_0u\|} \leq C_0. 
$$
Since $u_n$ is not in $W_c$, the sequence $p \mapsto e^{\theta_{\R,
    p}(b)} \|b^{-1}_p \cdots b^{-1}_0 u_n\|$ is not bounded above. For
$n$ large enough, there is therefore an integer $p_n$ such that 
\eq{eq: 7.14}{
\frac{e^{-M_0}\vre}{r_0C_0} \leq e^{\theta_{\R, p_n(b)}} \|b_{p_n}^{-1}
  \cdots b^{-1}_0 u_n \| \leq \frac{e^{-M_0}\vre}{r_0}
}
where $M_0 = \sup \tau$. We choose a rational $\ell_n$ such that $p_n
= p_{\ell_n}(c)$. This is possible since $\tau$ is strictly positive. 

Hence, since $\alpha < \frac{1}{4}$, we can choose an element $a = a_n
\in B$ such that it simultaneously belongs to the sets given by
\equ{eq: 7.7} and \equ{eq: 7.8}, \equ{eq: 7.12} and \equ{eq: 7.13}
with $\ell = \ell_n$, $u=u_n$ and $\eta = \eta_n \to 0$ and such that 
\eq{eq: 7.15}{
f(h_{\ell_n,c,x}(a_n)) = f(c,x) \text{ and } f(h_{\ell_n, c,
  y_n}(a_n)) = f(c,y_n).
}
Up to passing to a subsequence, we have 
\begin{enumerate}
\item
The sequence $(c'_n , x'_n) = h_{\ell_n,c,x}(a_n)$ has a limit
$(c',x') \in K$, 
\item
The sequence $(c'_n, y'_n) = h_{\ell_n,c,y_n}(a_n)$ has a limit in
$K$, and 
\item
the limit of the drift vector $w = \lim_{n \to \infty}
F_{\ell_n,c}(a_n)u_n$ exists, is nonzero, is of norm at most
$e^{-M_0}\vre$, and belongs to $V_{c'}$. 
\end{enumerate}
We then deduce, by passing to a limit in \equ{eq: 7.15}, since all the
limits considered have their values in $K$ and $L$, and since $f$ is
continuous on these sets, 
$$
f(c',x') = \lim_{n \to \infty} f(c'_n, x'_n) = \lim_{n \to \infty}
f(c,x) = f(c,x),
$$
$$
f(c', y') = \lim_{n \to \infty} f(c'_n, y'_n) = \lim_{n \to \infty}
f(c,y_n) \text{ and } y' = \exp(w)x'.
$$ 
In addition, if we let $v \in V_0$ be the nonzero vector $v =
D^{-1}_{c'}(w)$, we have 
$$
\|v \| \leq \vre \text{ and } (c', y') = \Phi_v(c',x'),
$$
which is the sought-for conclusion. 
\end{proof}

\subsection{Stabilizers of conditional measures}
{\em We will make explicit the information furnished by the drift
  argument, regarding the horocyclic conditional measures
  $\sigma(c,x)$, for $\beta^{\tau, X}$-a.e. $(c,x) \in B^{\tau, X}$.}

We introduce the connected stabilizers of the measures $\sigma(c,x)$
and their class $R^*_+ \sigma(c,x)$ modulo normalization:
$$
J(c,x) = \{ v\in V_0 : t_{v*}\sigma(c,x) = \sigma(c,x) \}_0,
$$
$$
J_1(c,x) = \{v \in V_0 : t_{v*}\sigma(c,x) \simeq \sigma(c,x) \}_0.
$$
These are closed subgroups and hence vector subspaces of $V_0$. 

\begin{prop}\name{prop: 7.4}
In the two cases of \S \ref{subsec: 6.1}, for $\beta^{\tau,
  X}$-a.e. $(c,x) \in B^{\tau, X}$, we have 
\begin{itemize}
\item[a)]
$J_1(c,x) \neq \{0\}$,
\item[b)]
$J(c,x) = J_1(c,x).$
\end{itemize}
\end{prop}
\begin{proof}
a) We will show, for $\beta^{\tau, X}$-a.e. $(c,x)$ and every
$\vre>0$, the stabilizer of $\sigma(c,x)$ modulo normalization
contains a nonzero vector of norm at most $\vre$. 

By Lemma \ref{lem: 6.11}, there is a Borel subset $E \subset
B^{\tau, X}$ such that $\beta^{\tau, X}(E^c) =0$ and such that, for
every $v \in V_0$ and $(c',x') \in E$ such that $\Phi_v(c',x') \in E$,
we have 
\eq{eq: 7.16}{
t_{v*}\sigma(\Phi_v(c',x')) \simeq \sigma(c',x').
}
By Corollary \ref{cor: 6.13}, the function $\sigma$ is
$\QQ_\infty^{\tau, X}$-measurable. The drift (Proposition \ref{prop:
  7.1}) applied to this set $E$ and this function $f = \sigma$
produces, for $\beta^{\tau, X}$-a.e. $(c,x) \in B^{\tau, X}$ and every
$\vre>0$, a nonzero vector $v \in V_0$ of norm at most $\vre$ and an
element $(c',x')$ of $E$ such that $\Phi_v(c',x')$ is also in $E$ and
such that 
$$
\sigma(\Phi_v(c',x')) \simeq \sigma(c',x') \simeq \sigma(c,x). 
$$
By applying \equ{eq: 7.16} to this element $(c',x')$, we find 
$$
t_{v*}\sigma(\Phi_v(c',x')) \simeq \sigma(c',x')
$$
and hence 
$$
t_{v*}\sigma(c,x) \simeq \sigma(c,x).
$$
The vector $v$ is indeed in the stabilizer of $\sigma(c,x)$ modulo
normalization. The stabilizer is non-discrete and closed. It thus
contains a nonzero linear subspace of $V_0$. 

b) For $\beta^{\tau, X}$-a.e. $(c,x) \in B^{\tau, X}$, there is a
linear form $\alpha(c,x) \in J_1(c,x)^*$ such that, for any $v \in
J_1(c,x)$, 
$$
t_{v*}\sigma(c,x) = e^{\alpha(c,x)(v)} \sigma(c,x).
$$
We wish to show $\alpha=0$. Lemma \ref{lem: 6.12} implies, for
$\beta^{\tau, X}$-a.e. $(c,x) \in B^{\tau, X}$, the equality
$J_1(T_\ell^{\tau, X}(c,x)) = J_1(c,x)$ and, for every $\ell \geq 0$,
the equality of linear forms on $J_1(c,x)$:
\eq{eq: 7.17}{
\alpha(T_\ell^{\tau, X}(c,x)) = e^\ell \alpha(c,x),
}
from which it follows, after applying the Poincar\'e recurrence
theorem, that $\beta^{\tau, X}$-almost surely, $\alpha =0$. 
\end{proof}

\subsection{Disintegration of $\nu_b$ along the stabilizers}
{\em In this section we will disintegrate the limit measures $\nu_b$
  along the connected components of the stabilizers of the horocyclic
  conditional. We will find that the measures $\nu_{b,x}$ are
  invariant under a nontrivial unipotent group. }

We will begin by translating the fact that the stabilizers of the
conditional horocyclic  measures are not discrete into a statement which does
not involve the suspension $B^\tau$. 

For $\beta$-a.e. $b \in B$, and $\nu_b$-a.e. $x \in X$, we denote by
$\sigma_{b,x} \in \MM(V_b)$ the conditional measure at $x$ of $\nu_b$
for the action on $X$ of $V_b$ through the group $\exp(V_b)$ (see 
\S \ref{subsec: 4.1}), and we denote $V_{b,x} \subset V_b$ the
connected component of the stabilizer of $\sigma_{b,x}$ in $V_b$. 

\begin{prop}\name{prop: 7.5}
In the two cases of \S \ref{subsec: 6.1}, for $\beta^X$-a.e. $(b,x)
\in B^X$, we have $\sigma_{b,x} \simeq b_{0*} \sigma_{T^X(b,x)}, \,
V_{b,x} = b_0\left(V_{T^X(b,x)} \right)$ and $V_{b,x} \neq 0.$ 
\end{prop}

\begin{proof}
The first equality follows from the equalities, for $\beta$-a.e. $b
\in B$, $\nu_{Tb} = (b^{-1}_0)_* \nu_b$ and, for every $x \in X$ and
$v \in \mathfrak{g}$, 
$$
T^X(b, \exp(v)x) = (Tb, \exp(b^{-1}_0 v)b^{-1}_0x).
$$
The second equality follows. 

The fact that $V_{b,x}$ is nonzero follows from Proposition \ref{prop:
  7.4} and the equality, for $\beta^{\tau, X}$-a.e. $(c,x) \in
B^{\tau, X}$, $V_{b,x} = R(s(\xi(b))m)(J(c,x))$, where $c= (b,k,m).$ 
\end{proof}

The disintegration of $\beta^X$ along the map $(b,x) \mapsto (b,
V_{b,x})$, or, what will turn out to be the same, the disintegration
for $\beta$-a.e. $b$ of $\nu_b$ along the map $x \mapsto V_{b,x}$, can
be written as 
$$
\nu_b = \int_X \nu_{b,x} d\nu_b(x)
$$
where, for $\beta^X$-a.e. $(b,x) \in B^X$, the probability measure
$\nu_{b,x}$ on $X$ is supported on the fiber $\{x' \in X: V_{b,x'} =
V_{b,x} \}$. 

\begin{prop}\name{prop: 7.6}
In the two cases of \S \ref{subsec: 6.1}, for
$\beta^{\tau,X}$-a.e. $(b,x) \in B^X$, the probability measure
$\nu_{b,x}$ is $V_{b,x}$-invariant and has the equivariance property
$\nu_{b,x} = b_{0*} \nu_{Tb, b^{-1}_0x}. $ 
\end{prop}
\begin{proof}
The first assertion follows from Proposition \ref{prop: 4.3}. 

The second assertion follows from the equality $\nu_b =
b_{0*}\nu_{Tb}$, from Proposition \ref{prop: 7.5}, and from the
disintegration of measures. 
\end{proof}

\section{Applications}
{\em In this chapter we conclude the proof of Theorems \ref{thm: 1.1}
  and \ref{thm: 1.3} and their corollaries.}

\subsection{Invariance of stationary measures}\name{subsec: 8.1}
{\em We keep the notations of \S \ref{subsec: 6.1} and we conclude
  this section with the classification of stationary measures on $X$.}

\begin{prop}\name{prop: 8.1}
In the two cases of \S \ref{subsec: 6.1}, the probability measure
$\nu$ is the Haar measure on $X$. 
\end{prop}

In order to deduce this from Proposition \ref{prop: 7.6}, We will need
the following lemma. Let $\alpha \in \PP(X)$. 

In the first case of \S \ref{subsec: 6.1}, we denote by $S_\alpha$ the
connected component of the identity in the stabilizer of $\alpha$ in
$G$, with respect to the action by translations on $X = G/\Lambda$. 

In the second case of \S \ref{subsec: 6.1}, we denote by $S_\alpha$
the connected component of the identity in the stabilizer of  $\alpha$
in $\R^d$ with respect to the translation action on $X = \TT^d$. 

In both cases, we set 
$$
\FF = \{\alpha \in \PP(X) : S_\alpha \neq \{1\} \text{ and } \alpha
\text{ is supported on one } S_\alpha\text{-orbit}\},
$$
and endow this collection with the weak-* topology. 

We note that the group $G$ acts naturally on $\FF$. Denote by $\nu_0$
the Haar measure on $X$. Then $\nu_0$ is an element of $\FF$. 

\begin{lem}\name{lem: 8.2}
In both cases of \S \ref{subsec: 6.1}, the only $\mu$-stationary Borel
probability measure $\eta$ on $\FF$ is $\delta_{\nu_0}$. 
\end{lem}
\begin{proof}
We can suppose that $\eta$ is $\mu$-ergodic. We will distinguish the
two cases:

{\bf First case of \S \ref{subsec: 6.1}.} In this case we have $X =
G/\Lambda$. 

By \cite[Thm. 1.1]{ref 30}, the set $\GG$ of $G$-orbits in $\FF$ is
countable. 

The image $\bar\eta$ of $\eta$ in $\GG$ is  a $\mu$-stationary ergodic
probability measure, on a countable set. By Lemma \ref{lem: 8.3}, the
probability measure $\bar \eta$ has finite support. 

Since $\eta$ is $\mu$-ergodic, it is supported on a unique orbit
$G\alpha \cong G/G_\alpha \subset \FF$. By definition of $\FF$, the
group $G_\alpha$ is not discrete. Since $G_\alpha$ contains a lattice,
it is unimodular. By Proposition \ref{prop: 6.7}, $G_\alpha = G$. The
probability measure $\nu$ is thus equal to $\delta_{\nu_0}$. 

{\bf Second case of \S \ref{subsec: 6.1}.} In this case we have $X =
\TT^d$. 

We denote by $\GG$ the set of nontrivial tori in $X$ and for $Y \in
\GG$, we denote by $\FF_Y$ the set of measures which are translates of
the Haar probability measure on $Y$. The space $\FF$ is thus a
countable union of compact subsets $\FF_Y$. 

The image $\bar \eta$ of $\eta$ in $\GG$ is a $\mu$-stationary ergodic
probability measure on a countable set. By Lemma \ref{lem: 8.3}, it
has finite support $Y_1, \ldots, Y_n$ and $\Gamma$ permutes the
subspaces $V_1, \ldots, V_n$ which are the tangent directions of the
tori $Y_1, \ldots, Y_n$. 

Since the action of $\Gamma$ is strongly irreducible, we necessarily
have $V= V_1 = \cdots = V_n$, which is what we had to prove. 
\end{proof}

We will use the following classical result.
\begin{lem}\name{lem: 8.3}
Let $\Gamma$ be a group acting on a countable space $X$ and let $\mu$ be a
probability measure on $\Gamma$. Any $\mu$-stationary and
$\mu$-ergodic measure $\nu$ is $\Gamma$-invariant and finitely
supported. 
\end{lem}

\begin{proof}[Proof of Lemma \ref{lem: 8.3}]
Let $Y$ be the set of points of $X$ with maximal mass
(w.r.t. $\nu$). Then $Y$ is finite. The equality $\nu = \mu*\nu$
and the maximum principle 
imply that for $\mu$-a.e. $\gamma \in \Gamma$, $\gamma^{-1}Y\subset Y$
and hence $\gamma^{-1} Y =Y.$ Since $\nu(Y)>0$ and $\nu$ is
$\mu$-ergodic, $\nu(Y)=1$. 
\end{proof}

\begin{proof}[Proof of Proposition \ref{prop: 8.1}]
By Proposition \ref{prop: 7.5}, the fruit of our efforts, for
$\beta^X$-a.e. $(b,x) \in B^X$, the subgroups $V_{b,x}$ are
nontrivial. 

The principal interest in  the set $\FF$ is that {\em it contains all
  of the probability measures invariant and ergodic under a connected
  nontrivial unipotent subgroup.} This results from Ratner's work
\cite{ref 30} in the first case and is elementary in the second case. 

For $\beta^X$-a.e. $(b,x) \in B^X$, the decomposition of $\nu_{b,x}$
into $V_{b,x}$-ergodic components can thus be written simultaneously
in the form 
\eq{eq: 8.1}{
\nu_{b,x} = \int_X \zeta(b,x') d\nu_{b,x}(x'),
}
where $\zeta: B^X \to \FF$ is a $\BB^X$-measurable map such that, for
$\beta^X$-a.e. $(b,x) \in B^X$, the restriction of $\zeta$ to the
fiber $\{(b,x'): V_{b,x'} = V_{b,x}\}$ is constant along the
$V_{b,x}$-orbits. 

The uniqueness of the ergodic decomposition, and Propositions
\ref{prop: 7.5} and \ref{prop: 7.6}, prove that, for
$\beta^X$-a.e. $(b,x) \in B^X$, we have 
\eq{eq: 8.2}{
\zeta(b,x) = (b_0)_* \zeta(T^X(b, x)). 
}
By Lemma \ref{lem: 3.2}(e), the image probability measure $\eta =
\zeta_* \beta^X$ is therefore a $\mu$-stationary probability measure
on $\FF$. By Lemma \ref{lem: 8.2}, this probability measure is the
Dirac mass on $\nu_0$. In other words, $\zeta(b,x)$ is
$\beta^X$-almost surely equal to $\nu_0$, so that $\nu = \nu_0$. 
\end{proof}

\begin{proof}[Proof of Theorems \ref{thm: 1.1} and \ref{thm: 1.3}]
Recall that, in the second case, we have denoted by $G$ the Zariski
closure of $\Gamma_\mu$ in $\SL(d,\R)$. Lemma \ref{lem: 8.5} below
shows that $G$ is also semi-simple. 

In both cases, Lemma \ref{lem: 8.4} below makes it possible to assume
that $G$ is a semi-simple noncompact Lie group. One can then apply
Proposition \ref{prop: 8.1} to conclude that $\nu$ is $G$-invariant. 
\end{proof}

We have used above the following two easy lemmas. 

\begin{lem}\name{lem: 8.4}
Let $K$ be a metrizable compact group acting in Borel fashion on a
Borel space $X$, and let $\mu$ be a Borel probability measure on
$K$. Then any $\mu$-stationary Borel probability measure $\nu$ on $X$
is invariant under the group $\Gamma_\mu$ generated by the support of
$\mu$. 
\end{lem}

\begin{proof}
By Varadarajan's theorem \cite[Prop. 2.1.19]{ref 35}, we may space
that $X$ is compact and that the action is continuous. We may also
suppose that $\nu$ is $\mu$-ergodic. It is then supported on a unique
$K$-orbit $Kx_0$. We can therefore consider $\nu$ to be an
$H$-invariant measure on $K$, for the action of $H$ on the right,
where $H$ is the stabilizer of $x_0$. This lifted probability measure
is also $\mu$-stationary. It remains to treat the case $X = K$. 

Up to convolving $\nu$ on the right by an approximate identity, we can
suppose that $\nu$ is absolutely continuous with respect to Haar
measure, with a continuous density. We can thus think of $\nu$ as an
element of $L^2(K)$ satisfying $\mu * \nu = \nu$. But in a Hilbert
space, the average of vectors of a fixed norm has norm strictly
smaller, unless the vectors being averages are equal to each
other. This proves that $\nu$ is $\Gamma_\mu$-invariant. 
\end{proof}

\begin{lem} \name{lem: 8.5}
Let $\Gamma$ be a subsemigroup of $\SL_d(\Z)$ which acts strongly
irreducibly on $\R^d$. Then its Zariski closure $G$ in $\SL(d, \R)$ is
a semisimple group. 
\end{lem}

\begin{proof}
We can suppose that $G$ is Zariski-connected. Since the representation
of $G$ on $\R^d$ is irreducible, $G$ is a reductive group. Since $G$
is made of matrices of determinant 1, its center $Z$ is compact. We
need to show that $Z$ is finite. 

Suppose by contradiction that $Z$ is infinite. The commutant of $G$ in
$\End(\Q^d)$ is then an imaginary quadratic extension of $K$ of
$\Q$. We can then regard $\Q^d$ as a $K$-vector space. The determinant
map $g \mapsto \det_K(g)$ embeds $\Gamma$ in the group of units $U_K$ of
$K$. Since $U_K$ is finite, the determinant map also embeds $G$ in
$U_K$. Therefore $Z$ is finite, a contradiction. 
\end{proof}

\subsection{Invariant measures}
{\em In order to deduce the corollaries of our theorems, we need to
conveniently choose the measure $\mu$. }

\begin{proof}[Proof of Corollaries \ref{cor: 1.2}(a) and \ref{cor:
    1.4}(a)]
Since $G$ is simple, any Zariski dense subsemigroup $\Gamma$ contains
a finitely generated subsemigroup $\Gamma'$ which is also Zariski
dense.  Denote by $g_1 , \ldots, g_\ell$ a set of generators of
$\Gamma'$ and let $\mu = \frac{1}{\ell} (\delta_{g_1} + \cdots +
\delta_{g_\ell}) \in \PP(G)$. 

Let $\nu$ be a non-atomic probability measure on $X$ which is
invariant under $\Gamma$. Then it is $\mu$-stationary. By Theorem
\ref{thm: 1.1} it is $G$-invariant, as required. 
\end{proof}

\subsection{Closed invariant subsets}
{\em In order to prove corollaries \ref{cor: 1.2}(b) and \ref{cor:
    1.4}(b), we will need the following lemma.}

\begin{lem}\name{lem: 8.6}
In the two cases of \S \ref{subsec: 6.1}, the collection of finite
$\Gamma$-invariant subsets of $X$ is countable. 
\end{lem}

\begin{proof}
As before, we may suppose that $\Gamma$ is finitely generated. Since
$\Gamma$ has countably many finite-index subgroups, it suffices to
show that the points of $X$ which are fixed by some subgroup $\Delta$
of $\Gamma$ are isolated. The last assertion follows from the fact
that in any neighborhood of a fixed point, the linearization of the action of $\Delta$ is
its action on $V$, and since the action of $\Gamma$ on $V$ is strongly
irreducible, $\Delta$ does not have nonzero fixed vectors in $V$. 
\end{proof}

\begin{proof}[Proof of Corollaries \ref{cor: 1.2}(b) and \ref{cor:
    1.4}(b)] We may again suppose that $\Gamma$ is finitely
  generated. We then denote, just as in the proof of point (a), that
  $\mu$ is the probability measure given by $\mu = \frac{1}{\ell}
  (\delta_{g_1} + \cdots + \delta_{g_\ell})$, where $g_1, \ldots,
  g_\ell$ are a set of generators of $\Gamma$. Let $F$ be an infinite
  closed $\Gamma$-invariant subset of $X$. By Lemma \ref{lem: 8.6}, we
  can construct an increasing sequence $F_1 \subset F_2 \subset \cdots
  \subset F_i \subset \cdots $ of finite $\Gamma$-invariant subsets
  (possibly empty) of $X$, such that every finite $\Gamma$-invariant
  subset is contained in one of the $F_i$. Since $F$ is infinite, we
  can choose pairwise distinct points $x_1, x_2, \ldots$ of $F$ such
  that $x_i$ is not in $F_i$ for each $i$. 

By Proposition \ref{prop: 6.4}, regarding recurrence off of finite
subsets, there is a collection $(K_i)_{i \geq 0}$ of compact subsets
such that for each $i$, $K_i$ is contained in $F^c_i$ and such that
for all $j \geq 1$, there is an integer $M_j$ such that for $n \geq
M_j$ and $i \leq j$, 
\eq{eq: 8.3}{
(\mu^{*n} * \delta_{x_j})(K_i^c) \leq \frac{1}{i}.  
}
Setting $n_j = jM_j$, we introduce the Birkhoff-Kakutani averages 
\eq{eq: 8.4}{
\nu_j = \frac{1}{n_j} (\mu * \delta_{x_j} + \cdots + \mu^{*n_j}* \delta_{x_j}).
}
We have, for all $i \leq j$, 
\eq{eq: 8.5}{
\nu_j(K^c_i) \leq \frac{M_j}{n_j} +\frac{n_j - M_j}{n_j} \frac{1}{i} \leq\frac{2}{i}.
}
Condition \equ{eq: 8.5} ensures that any accumulation point of the
sequence $(\nu_j)$ for weak-* convergence of Borel probability
measures, is a probability measure which gives no mass to the subsets
$F_i, i \geq 1$. If $\nu_\infty$ is such an accumulation point,
$\nu_\infty$ is then a $\mu$-stationary Borel probability measure
satisfying $\nu_\infty(F) = 1$ and $\nu_\infty$ is non-atomic, by
Lemma \ref{lem: 8.3}. According to Theorems \ref{thm: 1.1} and
\ref{thm: 1.3}, $\nu_\infty$ is Haar measure. This implies the
required equality $F=X$. 
\end{proof}

\subsection{Equidistribution of finite orbits}
{\em The same arguments lead to a proof of equidistribution of finite
  orbits.}

\begin{proof}[Proof of Corollaries \ref{cor: 1.2}(c) and \ref{cor:
    1.4}(c)]
We may again suppose that $\Gamma$ is generated by the finite support
of the measure $\mu$. We will show that the sequence of
$\Gamma$-invariant measures 
$$
\nu_j = \frac{1}{\#X_j} \sum_{x \in X_j} \delta_x
$$
converges weak-* to the Haar probability measure on $X$. By point (a),
we just have to show that any weak limit $\nu_\infty$ of the sequence
$(\nu_j)$ is a probability measure which gives zero mass to finite
orbits. The proof relies on the phenomenon of recurrence off of finite
orbits. This is analogous to point (b) and we keep the notations $F_i$
and $K_i$. 

Since the finite $\Gamma$-orbits $X_j$ are distinct, we can suppose
after passing to a subsequence that for every $j \geq i$, we have
$\nu_j(F_i)=0$. Since $\nu_j$ is $\Gamma$-invariant, for any $n \geq
0$, we have $\mu^{*n} * \nu_j = \nu_j$ and therefore, as in (b), for
any $j \geq i, \, \nu_j(K^c_0) \leq \frac{1}{i}. $ We deduce that for
all $i \geq 0$, we have $\nu_\infty(K_i^C) \leq \frac{1}{i}$. This
implies that firstly, $\nu_\infty$ is a probability measure, and
secondly, that $\nu_\infty(F_i) =0$ for all $i$, and therefore that
$\nu_\infty$ is Haar measure. 
\end{proof}

\end{document}